\documentclass[12pt]{article}
\usepackage{amsmath,amssymb,amsthm,eucal}
\usepackage{color}
\usepackage[all]{xy}

\pdfoutput=1

\title{\vspace*{-1pc}%
  Curvature of differentiable Hilbert modules and Kasparov modules
}

\author{Bram Mesland\S$^*$, Adam Rennie\dag, Walter D. van Suijlekom\ddag
\thanks{email: 
\texttt{b.mesland@math.leidenuniv.nl},\,\texttt{renniea@uow.edu.au},\,\texttt{waltervs@math.ru.nl}
}
\\[3pt]
\S Mathematisch Instituut, Universiteit Leiden, The Netherlands
\\[3pt]
\dag School of Mathematics and Applied Statistics, University of Wollongong\\
Wollongong, Australia\\[3pt]
\ddag IMAPP – Mathematics,
Faculty of Science\\
Radboud University Nijmegen, The Netherlands
}

\advance\topmargin by -\headheight
\advance\topmargin by -\headsep
\evensidemargin=\oddsidemargin
\makeatletter
\def\section{\@startsection{section}{1}{\z@}{-3.5ex plus -1ex minus
  -.2ex}{2.3ex plus .2ex}{\large\bf}}
\def\subsection{\@startsection{subsection}{2}{\z@}{-3.25ex plus -1ex
  minus -.2ex}{1.5ex plus .2ex}{\normalsize\bf}}
\makeatother

\numberwithin{equation}{section} 

\theoremstyle{plain} 
\newtheorem{thm}{Theorem}[section]
\newtheorem{lemma}[thm]{Lemma}
\newtheorem{prop}[thm]{Proposition}
\newtheorem{corl}[thm]{Corollary}

\theoremstyle{definition} 
\newtheorem{defn}[thm]{Definition}

\theoremstyle{remark} 
\newtheorem{rmk}[thm]{Remark}

\DeclareMathOperator{\Dom}{Dom}   
\DeclareMathOperator{\End}{End}   
\DeclareMathOperator{\supp}{supp} 
\DeclareMathOperator{\tr}{tr}     
\newcommand{\alg}{\textnormal{alg}}
\newcommand{\Tex}{\textnormal}

\newcommand{\diag}{\textnormal{diag }}

\newcommand{\A}{\mathcal{A}}  
\newcommand{\B}{\mathcal{B}}  
\newcommand{\BB}{\mathbb{B}} 
\newcommand{\C}{\mathbb{C}}   
\newcommand{\D}{\mathcal{D}}  
\renewcommand{\d}{\mathrm{d}} 
\newcommand{\E}{\mathcal{E}}  
\renewcommand{\H}{\mathcal{H}}  

\newcommand{\J}{\mathcal{J}}  
\newcommand{\ox}{\otimes}     
\newcommand{\R}{\mathbb{R}}   
\newcommand{\X}{\mathcal{X}}  
\newcommand{\Z}{\mathbb{Z}}   

\newcommand{\stroke}{\mathbin|}     

\def\pairL_#1(#2|#3){{}_{#1}(#2\stroke#3)} 
\def\pairR(#1|#2)_#3{(#1\stroke#2)_{#3}} 
\def\scal<#1|#2>{\langle#1\stroke#2\rangle} 

\newbox\ncintdbox \newbox\ncinttbox 
	\setbox0=\hbox{$-$}
	\setbox2=\hbox{$\displaystyle\int$}
	\setbox\ncintdbox=\hbox{\rlap{\hbox
		to \wd2{\hskip-.125em \box2\relax\hfil}}\box0\kern.1em}
	\setbox0=\hbox{$\vcenter{\hrule width 4pt}$}
	\setbox2=\hbox{$\textstyle\int$}
	\setbox\ncinttbox=\hbox{\rlap{\hbox
		to \wd2{\hskip-.175em \box2\relax\hfil}}\box0\kern.1em}

\begin{document}

\maketitle

\vspace{-2pc}

\begin{abstract} In this paper we introduce the curvature of 
densely defined universal connections on Hilbert $C^{*}$-modules 
relative to a spectral triple (or unbounded Kasparov module), 
obtaining a well-defined 
curvature operator. Fixing the spectral triple, we find that 
modulo junk forms, the curvature only depends on the represented form 
of the universal connection. We refine our definition of curvature to factorisations 
of unbounded Kasparov modules.
Our refined definition recovers all the curvature data of a Riemannian submersion of compact manifolds, viewed as a $KK$-factorisation.







\end{abstract}

\tableofcontents

\parskip=4pt
\parindent=0pt

\addtocontents{toc}{\vspace{-1pc}}
\section*{Introduction}
This paper offers a new approach to defining and effectively computing
curvature of Hilbert modules and unbounded Kasparov modules. We will introduce a notion of curvature directly at the operator-algebraic level, thus sidestepping 
some
of the difficulties imposed by the absence of a differential graded algebra. Our approach does not rely on the heat kernel coefficient
analogy, and so our results differ from the recent work of \cite{CT,CF,CMcurv,FGK,GK,LeMo1,LeMo2}. Rather, 
we provide a complementary point of view. Numerous other
approaches to curvature have appeared independently in many
works \cite{AW,BGJ1,BGJ2,DHLS,DMM,Ros1}.

The usual theory of curvature of $\mathbb{Z}_2$-graded right modules $\X$ over associative 
algebras $\B$ 
relies, in its most algebraic formulation, on the existence of a
differential graded algebra $(\Omega^*(\B),\delta)$. The first step, 
existence of connections
$$
\nabla:\X\to\X\ox_\B\Omega^1(\B)\qquad \nabla(xb)=\nabla(x)b+\gamma(x)\otimes \delta(b),\quad x\in\X,\ b\in\B
$$
where $\gamma$ is the grading, was settled by Cuntz and Quillen \cite{CQ}. 
The curvature $R_\nabla$
of $(\X,\nabla)$ is then 
defined 
as the composition
\begin{equation}
\label{eq: intronasquare}
\nabla^{2}:\X_\B\stackrel{\nabla}{\to}\X\ox_\B\Omega^1(\B)\stackrel{\nabla\ox 1+1\ox\delta}{\longrightarrow}\X\ox_\B\Omega^2(\B).
\end{equation}
Thus as soon as we have a differential graded algebra and a connection we obtain an
endomorphism-valued two-form $\nabla^{2}$. The full details of this approach
in noncommutative geometry, pioneered
by Connes and Rieffel \cite{CR}, appear in the book of Connes \cite[Sect. VI.1]{BRB} (see also \cite[Sect. 7.2]{Landi}).



Instead, the route we take to curvature is inspired by tools and ideas from 
unbounded $KK$-theory, and the unbounded version of the internal Kasparov product in particular. As we will see, the above algebraic notion of curvature appears naturally in this functional analytical framework. Let us sketch the main idea. 

Suppose we are given two (suitably differentiable) unbounded 
$KK$-cycles $(\mathcal A,X,S)$ and $(\mathcal B,Y,T)$ and a (suitable) connection $\nabla$ on $X$. We may consider an unbounded representative of the internal Kasparov product given by the essentially self-adjoint and regular operator $S \otimes 1 + 1 \otimes_\nabla T$, defined on the appropriate  domain in $X \otimes_{B} Y$ (see 
\cite{BMS,KaLe2,LM,Mes,MR} for more details). Our definition of curvature is given by the  formula 
\begin{equation}
  \label{correspondencecurve-intro}
R_{(S,\nabla_{T})}=(S\otimes 1+1\otimes_{\nabla}T)^{2}-(S^{2}\otimes 1
+1\otimes_{\nabla}T^{2}).
\end{equation}
We will make precise sense of this unbounded operator on $X \otimes_{B} Y$ in due course, but let us highlight some of its features:
\begin{itemize}
\item the curvature $R_{(S,\nabla_{T})}$ can be interpreted as a measure of the ``defect'' of the internal Kasparov product to respect the taking of squares of the operators $S$ and $T$. This is in line with the notion of curvature for linear maps used in cyclic theory (see \cite[page 255]{CQ}).
  \item it vanishes for a direct product of spaces: for the {\em external} Kasparov product we have for the tensor sum (on graded modules):  
    $$
(D_1 \otimes 1 + 1 \otimes D_2)^2 - D_1^2 \otimes 1 - 1 \otimes D_2^2 = 0
    $$
  \item the ``geometric'' information described by $R_{(S,\nabla_{T})}$ is only accessible at the unbounded level, thus forming a refinement of the topological information described at the level of (bounded) $KK$-theory. 
  \end{itemize}

Our main task is now to make sense of formula \eqref{correspondencecurve-intro} so let us see what it says {\em algebraically}. Upon expanding the brackets we observe that the curvature can be understood in terms of the following two operators
\begin{align*}
(1\otimes_{\nabla}T)^2-1\otimes_{\nabla}T^{2}, \qquad 
  \left[ S \otimes 1 , 1 \otimes_\nabla T \right ].
\end{align*}
The approach we will take is to first make sense of the above two operators and then define $R_{(S,\nabla_{T})}$ in terms of them in Section \ref{sect:curv-corresp}. Intriguingly, the well-definedness of the operator $(1\otimes_{\nabla}T)^2-1\otimes_{\nabla}T^{2}$ in Definition \ref{def: curv-operator} relies heavily on the existence of a relative $S$-bound on the commutator $\left[ S \otimes 1 , 1 \otimes_\nabla T \right ]$ ({\em cf.} Definition \ref{wac} below), which in turn is a sufficient condition
for representing the Kasparov product (see \cite{Kuc, LM}). Moreover, it turns out that the operator $(1\otimes_{\nabla}T)^2-1\otimes_{\nabla}T^{2}$ is of interest in itself, and we will call it the curvature operator of the ($C^2$-) connection $\nabla$ on the module $X$. 

This terminology is justified by the result (Theorem \ref{thm: Ristwoformdetermineduptojunk}) that $(1\otimes_{\nabla}T)^2-1\otimes_{\nabla}T^{2}$ is given by a represented square $\pi_T(\nabla^2)$ of a universal connection $\nabla$ as we now explain. 

Unbounded Kasparov modules like $(\B,Y,T)$ 
provide the basic geometric objects of non-commutative geometry.
One feature of the Kasparov module $(\B,Y,T)$ is the  bimodule of differential
one-forms \cite[Chapter VI]{BRB} 
$$
\Omega^1_T(\B):={\rm span}\{a[T,b]:\,a,\,b\in\B\}.
$$
The bimodule $\Omega^1_T(\B)$ 
consists of operators on the Hilbert $C^{*}$-module $Y$. 
The differential graded algebra of universal differential forms 
$\Omega^{*}_{u}(\B)$ can then be represented as operators on 
$Y$ by taking products, but the image of this representation 
does not carry the structure of a differential graded algebra anymore. 

The obstruction to defining a differential is the existence 
of \emph{junk-forms} \cite[Chapter VI]{BRB}. 
Although quotienting out the junk forms yields a 
differential graded algebra, it can no longer be represented on $Y$. 
Our represented curvature
\[
\pi_{T}(\nabla^{2})=1\otimes_{\nabla}T^{2}-(1\otimes_{\nabla}T)^{2},
\]
is well-defined up to junk forms, and so connects to the  
existing literature on curvature of connections. 
Of course the challenge is to make sense of this square, 
which we do in Section \ref{sect:curv-nabla2}.

It is also useful to illustrate our notion of curvature for finitely-generated projective modules $X$ over $\mathcal B$ (see Section \ref{sect:fgp} below for full details). 
So, consider computing the Kasparov
product $(\C,X_B,0)\ox_B(\B,H,D)$. In this case, a 
smooth submodule $\X_\B\subset X_B$ is guaranteed to exist. Then realising $\X\cong p\B^N$
for some projection $p\in M_N(\B)$, all compatible connections 
are of the form $p\circ \d +A$
for 
$A\in \X\ox_\B \Omega^1_D(\B)\ox_\B\X^*$ an endomorphism-valued
one-form. For any (Hermitian) connection
on the module $\X$ we obtain a representative $(\C,X\ox_BH,1\ox_\nabla D)$
of the Kasparov product with operator
\begin{equation}
1\ox_\nabla D:\,\X\ox_\B \Dom(D)\to X\ox_B H\quad
1\ox_\nabla D(x\ox\xi):=\gamma(x)\ox D\xi+ \nabla_{D}(x)\xi.
\label{eq:op-prod}
\end{equation}
The key observation is then that the curvature operator is given by
\begin{equation}
(1\otimes_\nabla D)^2-1\otimes_\nabla D^2=p[D,p][D,p]p+A^2+\d A.
\label{eq:Bram-is-a-good-boy}
\end{equation}
Here $\d A=\pi_{D}(\delta A_{u})$ indicates an operator defined through choosing a lift $A_{u}$ of $A$ to the universal calculus and is independent of the choice of lift 
up to junk forms. Nevertheless, the left hand side of
Equation \eqref{eq:Bram-is-a-good-boy}
is a well-defined, direct and constructive way of representing the curvature
of a module $\X_\B$: all we require is
the differential structure provided by 
a spectral triple or unbounded Kasparov module.

Going beyond finitely generated modules, for countably 
generated $C^{*}$-$B$-modules we not only require a differentiable 
structure induced from a differentiable structure on $B$, but we also need to fix a regular operator $S$ on $X$. This operator should be thought of
as defining a vertical differential structure on $X$. Hilbert $C^{*}$-modules are generalisations of continuous fields of Hilbert spaces, and such fields are trivial
if and only if they are locally trivial. Thus, in order to detect nontrivial topological content, working at the continuous level will not suffice. Differential structures on
continuous fields of Hilbert spaces are not naturally given, and need to be prescribed. 
This phenomena requires us to talk about both horizontal and vertical
differentiability, and as already noted, 
these considerations are compatible with $KK$-factorisation.

Examples where curvature appears in the context of unbounded Kasparov theory is in the factorisation of Dirac operators on Riemannian submersions and $G$-spectral triples   \cite{BMS, CaMes, KS16}. We will review and illustrate our notion of curvature for Riemannian submersions in Section \ref{sect:riem-subm}. As we will see, the curvature operator contains the
information of the second fundamental form of the Riemannian submersion, the mean curvature associated to it, as well as the curvature of the metric connection used in the Kasparov product. Note once again that all this geometric information becomes available only at the unbounded level, thus refining the topological information present at the level of bounded $KK$-theory. 

Section \ref{sec:two} outlines both the algebraic and analytic aspects of differential forms. Section \ref{Hilbcurve} outlines the analysis required to make sense of the curvature operator, and especially second derivatives. Section
\ref{sect:fgp} outlines the consequences for finitely generated projective modules, Section \ref{sec:graaaaaaaaaa} gives results for the special case of Grassmann 
connections,
and in Section \ref{sect:riem-subm} we outline the application to Riemannian submersions.

\subsection*{Notation}
All algebras denoted by symbols $A,B,C$ are assumed to be \emph{unital} and \emph{trivally graded}. All modules denoted 
by symbols $X,Y$ are assumed to be $\mathbb{Z}/2$-graded, with grading operator $\gamma_{X},\gamma_{Y}$ or simply $\gamma$. 
Consequently, algebras of operators on such modules are $\mathbb{Z}/2$-graded as well. Homomorphisms between graded algebras are
assumed to respect the grading. All commutators $[a,b]$ are graded commutators, which, for homogenous elements $a,b$ with degrees $\partial a$, $\partial b$ respectively 
are defined
to be 
\[
[a,b]:=ab-(-1)^{\partial a\partial b}ba,
\] 
and extended by linearity. Sometimes we write 
$[\cdot,\cdot]_+$ for anticommutators for emphasis.

We use various completed tensor products. The algebraic tensor product of modules will be denoted $ \otimes^{\rm alg}$, the completed tensor product of Hilbert $C^{*}$-modules will be denoted $\otimes$, the completed projective tensor product of locally convex vector spaces will be denoted $\widehat \otimes$, and Haagerup tensor product of operator spaces will be denoted $\otimes^h$.


%

{\bf Acknowledgements}
The authors acknowledge the support of 
the Erwin Schr\"{o}dinger Institute during the Thematic Programme \emph{Bivariant $K$-theory in geometry and physics} in November 2018, 
when a substantial part of this work was conducted. 
AR and BM thank the Gothenburg Centre for Advanced Studies in Science and Technology 
for funding and the University of Gothenburg and Chalmers University of Technology 
for their hospitality in November 2017 when this project took shape. We thank Alan Carey, Jens Kaad 
and Giovanni Landi for inspiring feedback in the course of this work. BM thanks Matthias Lesch for
numerous conversations related to some of the technical aspects of this work.

\section{Universal differential forms for $C^{2}$-Kasparov modules}
\label{sec:two}

This section develops the necessary tools to talk about 
universal differential forms in the context of unbounded Kasparov modules. We do not need forms of all 
degrees: in order to talk about curvature, degrees one and two 
suffice and we restrict attention to these degrees. 

We carefully capture the corresponding $C^1$-and $C^2$-topologies in terms of suitable operator $\ast$-algebras (see \cite{BKM}), which we will first introduce.

\subsection{Operator $*$-algebras and differential structures}
Throughout this section we will develop our approach to differential structure
relative to a fixed unbounded
\emph{Kasparov module} $(\mathcal{B},Y,T)$. This Kasparov module consists of a 
complex $*$-algebra $\mathcal{B}$, a $\mathbb{Z}_{2}$-graded Hilbert 
module over a $C^*$-algebra $C$, and an 
odd self-adjoint operator $T:\Dom T\to Y$. This data satisfies the requirements: 
\begin{enumerate}
\item there is an \emph{injective} $*$-representation $\mathcal{B}\to \mathbb{B}(Y)$;  
\item $T:\Dom T\to H$ is a self-adjoint operator 
and $b(T\pm i)^{-1}\in\mathbb{K}(Y)$ for all $b\in\mathcal{B}$;
\item for all $b\in\mathcal{B}$ it holds that 
$b:\Dom T\to \Dom T$ and $[T,b]$ extends 
to a bounded operator. 
\end{enumerate}
\begin{rmk}
We note that the property of locally compact resolvent in point $2.$
is {\em not used anywhere} in connection with the differential structure 
provided by $T$. So for the purposes of differential structure, we may
use any unbounded operator on the Hilbert module $Y$. In particular,
one may use indefinite Kasparov modules (non-commutative
analogues of pseudo-Riemannian geometries, \cite{DR}) to define differential structures.
\end{rmk}

We can always reduce to the case of operators on a Hilbert space.
Given a Kasparov module $(\B,Y_C,T)$, we may choose 
an injective Hilbert space representation $C\to B(H)$. Then 
we obtain the injective $*$-homomorphism 
$\mathbb{K}(Y)\to \mathbb{B}(Y\otimes_{C}H)$ by \cite[Proposition 4.7]{Lance} and since $\mathbb{B}(Y)=M(\mathbb{K}(Y))$ 
this extends to an injective $*$-homomorphism $\mathbb{B}(Y)\to\mathbb{B}(Y\otimes_{C}H)$. We may therefore consider $(\B, Y\otimes_{C} H, T\otimes 1)$.

%
%
We write $B$ for the $C^{*}$-closure of 
$\mathcal{B}$ in the norm it inherits as an 
algebra of operators on $\mathbb{B}(Y)$. An \emph{operator space} is a closed
subspace of $\mathbb{B}(H)$, an \emph{operator algebra} is an operator space 
that is closed under multiplication in $\mathbb{B}(H)$, and an \emph{operator $*$-algebra} is
an operator algebra that carries a completely isometric involution, \cite[Definition 1.4]{BKM}.

Consider 
the algebra representation
\[
\pi_{T}^{1}: \B\ni b\mapsto
\begin{pmatrix} b & 0 \\ [T,b] & b \end{pmatrix}
\in\mathbb{B}(Y\oplus Y).
\]
 The representation $\pi_{T}^{1}$ extends to matrices and satisfies the properties
 \[\|\pi^{1}_{T}(b)^{*}\|=\|\pi_{T}^1(b^{*})\|,\quad 
 \|b\|\leq \|\pi^{1}_{T}(b)\|,\quad b\in M_{n}(\B),
 \]
 and thus defines an operator $*$-algebra structure on the closure $\mathcal{B}_{1}$ of $\mathcal{B}$ in the norm
 \[
 \|b\|_{1}:=\|\pi_{T}^{1}(b)\|,
 \]
 which is compatible with the $C^{*}$-norm on $B$ in the sense that the inclusion $\mathcal{B}\to B$ is a completely contractive homomorphism of operator $*$-algebras.
 
We now present the additional $C^{2}$-condition the 
unbounded Kasparov module $(\mathcal{B},Y,T)$ is required to satisfy.
\begin{defn} 
\label{C2spec}
An unbounded Kasparov module $(\mathcal{B},Y,T)$ satisfies the 
\emph{$C^{2}$-condition} if there is a core 
$\mathcal{C}\subset\Dom T^{2}$ for $T^{2}$ 
such that for all $b\in\mathcal{B}$ 
we have 
\[
b:\mathcal{C}\to\Dom T^{2},
\]
and the densely defined operator
\[
[T^{2},b](T\pm i)^{-1}:(T\pm i)\mathcal{C}\to Y,
\]
extends to a bounded operator on $Y$.  
\end{defn}
Note that any core for $T^{2}$ is a core for $T$ 
and that $[T,b]$ extends to a bounded operator for all 
$b\in \mathcal{B}$ since $(\mathcal{B},Y,T)$ is an unbounded Kasparov module. 
 
The $C^{2}$-topology on $\mathcal{B}$ is defined 
as in \cite{Mes}. A concrete description as an operator 
$*$-algebra comes from the algebra representations
\begin{equation}
\label{diffreps}
\pi_{T}^{1}: a\mapsto
\begin{pmatrix} a & 0 \\ [T,a] & a \end{pmatrix}
\in\mathbb{B}(Y\oplus Y),\quad 
\pi_{T}^{2}:a \mapsto 
\begin{pmatrix} (T+i)a(T+i)^{-1} & 0 \\ [T^{2},a](T+i)^{-1} & a\end{pmatrix}
\in\mathbb{B}(Y\oplus Y).
\end{equation}

The representation $\pi^{2}_{T}$ extends to matrices but is not directly compatible 
with the $*$-structure (as in \cite{Mes}), so we define the operator $*$-norm
\[
\|b\|_{2}:=
\max\big\{\|\pi^{1}_{T}(b)\|,\|\pi^{2}_{T}(b)\|,\|\pi_{T}^{2}(b^{*})\|\big\},\quad b\in M_{n}(\B).
\]
The norm $\|b\|_{2}$ is realised concretely in the representation
\begin{equation}
\pi(a):=\pi_{T}^{1}(a)\oplus \pi^{2}_{T}(a)\oplus \pi^{2}_{T}(a^{*})^{*}.
\end{equation}
This gives the completion $\mathcal{B}_{2}$ of $\B$ in the 
norm $\Vert\cdot\Vert_2$ the structure of an 
operator $*$-algebra \cite{BKM}. By construction, the inclusions 
$\mathcal{B}_{2}\to\mathcal{B}_{1}\to B$ are completely contractive 
operator $*$-algebra homomorphisms. We now discuss several 
realisations of  bimodules of universal differential forms over $\mathcal{B}$.

\begin{rmk}
The norm $\Vert\cdot\Vert_2$ gives an analogue of a $C^2$-norm,
and to define this norm we needed the additional smoothness of the Kasparov module described in Definition \ref{C2spec}. In this light,
an unbounded Kasparov module is a $C^1$-Kasparov module,
having enough smoothness to define the ``$C^1$-norm'' $\Vert\cdot\Vert_1$. 
\end{rmk}

In the sequel we will make frequent use of the Haagerup tensor product for operator spaces (see \cite{BleLeM}). Given two operator spaces $X$ and $Y$ their Haagerup tensor product is the completion of $X\otimes^{\alg}Y$ in the norm
\[\|z\|_{h}^{2}:=\inf\left\{\left\|\sum x_{i}x_{i}^{*}\right\|\left\|\sum y_{i}^{*}y_{i}\right\|:z=\sum x_i\otimes y_{i}\right\},\]
which is denoted $X\otimes^{h}Y$ and can be shown to be an operator space again. It follows quite directly from the definition 
that for closed subspaces $A,B\subset \mathbb{B}(H)$, the multiplication map
\[m:A\otimes^{h}B\to \mathbb{B}(H),\quad a\otimes b\mapsto ab,\]
is completely contractive, and this property motivates the definition of the Haagerup norm (see \cite[Theorem 2.3.2]{BleLeM}) .
For our purposes the following characterisation of elements in $X\otimes^{h}Y$ is of crucial importance.
\begin{prop}[Proposition 1.5.6 of \cite{BleLeM}] 
\label{Haaconvergent}
Let $X$ and $Y$ be operator spaces, and $X\otimes^hY$ the Haagerup tensor product.
\begin{enumerate}
\item Let $z\in X\otimes^{h}Y$ with $\|z\|_{h}< 1$. Then $z$ can be written as a norm convergent series
$z=\sum x_{k}\otimes y_{k}$, where $\sum_{k} x_{k}x_{k}^{*}$ and $\sum y_{k}^{*}y_{k}$ are norm convergent series of norm $< 1$ in $X$ and $Y$ respectively;
\item If the sequences $(x_{k})\subset X$ and $(y_{k})\subset Y$ are such that $\sum_{k}y_{k}^{*}y_{k}$ is norm convergent and $\sup_{N}\|\sum_{|k|\leq N}x_{k}x_{k}^{*}\|<\infty$, then
$\sum x_{k}\otimes y_{k}$ is norm convergent in $X\otimes^{h}Y$. 
\end{enumerate}
\end{prop}
If $\B$ is an algebra and $X$ is a right- and $Y$ a left- $\B$ module, the \emph{Haagerup module tensor product} $X\otimes^{h}_{\B}Y$ is the quotient 
of $X\otimes^{h}Y$ by the closed subspace generated by $xb\otimes y-x\otimes by$ (see \cite[Section 3.4.2]{BleLeM}). Hilbert $C^{*}$-modules carry a natural operator space structure inherited from the embedding into their linking algebra. The following theorem characterizes the Haagerup module tensor product for Hilbert $C^{*}$-modules.
\begin{thm}[Theorem 4.3 of \cite{Blecher}]
\label{thm: Blecher}
Let $X$ and $Y$ be Hilbert $C^{*}$-modules over $C^{*}$-algebras $B$ and $C$ respectively, and suppose $B\to \mathbb{B}(Y)$ is a $*$-homomorphism. Then the Hilbert $C^{*}$-module tensor product $X\otimes_{B}Y$ is completely isometrically isomorphic to the Haagerup module tensor product $X\otimes^{h}_{B}Y$.
\end{thm}
We will use Theorem \ref{thm: Blecher} result freely in the sequel. We use the symbol $\otimes_{B}$ for the $C^{*}$-module tensor product, $\otimes^{h}_{B}$ for the Haagerup module tensor product and $\otimes_{\B}^{\alg}$ for the balanced algebraic tensor product.
\subsection{Universal and represented 
differential forms for $C^{2}$-spectral triples} 
\label{subsec:diff-forms}
We wish to define bimodules of universal 1-forms and 2-forms 
associated to a $C^{2}$-Kasparov module $(\mathcal{B},Y,T)$. 
To this end we use the Haagerup tensor product for the 
unital operator $*$-algebras $\mathcal{B}_{2},\mathcal{B}_{1}$ and $B$. 
To define universal one-forms we need to consider the kernel of the multiplication map $m:\,B\times B\to B$ restricted to suitable subalgebras.
We define three spaces of \emph{universal one-forms} 
for the Kasparov module $(\mathcal{B},Y,T)$. 
\begin{align*}
\Omega^{1}_{u}(B,\B_{2})
&:=\ker\big(B\otimes^{h}\B_{2}\xrightarrow{m} B\big)\\
\Omega^{1}_{u}(B, \B_{1})
&:=\ker\big(B\otimes^{h}\B_{1}\xrightarrow{m} B\big)\\
\Omega^{1}_{u}(\B_{1},\B_{2})
&:=\ker\big(\B_{1}\otimes^{h}\B_{2}\xrightarrow{m} \B_{1}\big),
\end{align*} 
with $m$ the multiplication map. As $m$ is a complete 
contraction on each of these spaces, the respective modules 
of forms are operator bimodules for the respective algebras.

We  denote elements of $\Omega^1_u(B,\B_1)$ (for instance)
as $\sum_ia_i\delta(b_i)$ where $\delta(b)=1\ox b-b\ox 1$. Here we
should take $a_i\in B$ and $b_i\in\B_1$, with similar descriptions of 
the other bimodules of one-forms. 

\subsubsection{Rough algebraic outline}

Let us give a brief algebraic sketch of what we need our forms to do,
so that the purpose of the analysis to follow is clear.
There is a map
$$
\pi_T:\,\Omega^1_u(B,\B_1)\to \BB(Y),\quad \pi_T(a\delta(b))=a[T,b].
$$
The range is denoted $\Omega^1_T(B,\B_1)$, and these are  called the
represented one forms. Restricting $\pi_T$ to the other bimodules
gives different spaces of one-forms.
To discuss curvature we need to 
consider two-forms. Universally we have a few options, for instance,
$$
\Omega^2_u(B,\B_2)
:=\Omega^1_u(B,\B_1)\ox^h_{\B_1}\Omega^1_u(\B_1,\B_2).
$$
The common factor of $\B_1$ allows us to use the Leibniz rule
$$
a\delta(b_1)b_2\delta(c)=a\delta(b_1b_2)\delta(c)-ab_1\delta(b_2)\delta(c),
\quad a\in B,\ b_1,\,b_2\in\B_1,\ c\in\B_2,
$$
to see that all two-forms can be represented as sums 
$\sum_ia_i\delta(b_i)\delta(c_i)$
for appropriate algebra elements.
The \emph{universal differential}  
$\delta:\Omega^{1}_u(\mathcal{B}_{1},\mathcal{B}_{2})
\to \Omega^{1}_{u}(B,\mathcal{B}_{1})\otimes^{h}_{\mathcal{B}_{1}}\Omega^{1}_{u}(\mathcal{B}_{1},\mathcal{B}_{2})$ 
is defined by
\begin{align*}
a\delta(b)=a\otimes b-ab\otimes 1 
&\mapsto (1\otimes a- a\otimes 1) \otimes (1\otimes b- b\otimes 1)
=\delta(a)\delta(b),
&\\
&=1\otimes a\otimes b-1\otimes ab\otimes 1 -a\otimes 1\otimes b+a\otimes b\otimes 1
\end{align*}
and satisfies $\delta^{2}(a)=0$ for all $a\in\mathcal{B}_{2}$. 

By declaring the
symbol $\delta$ to be odd, we obtain a $\Z_2$-grading on 
the various spaces of universal one-forms $\Omega^{1}_{u}$.
Since $B$ is trivially graded, all elements of $\Omega^{1}_{u}$ are odd.

Since the map $\pi_T$ is $B$-bilinear, and the Haagerup tensor product linearises
operator multiplication,
we can also represent our two-forms in $\mathbb{B}(Y)$ via

$$
m\circ(\pi_{T}\otimes \pi_{T}):\Omega^2_u(B,\B_2)\ni a\delta(b)\delta(c)
\mapsto \pi_T(a\delta(b))\pi_T(\delta(c))= a[T,b][T,c].
$$
The map $m\circ (\pi_T\otimes\pi_{T})$ is compatible with $*$-structures as well if we define
$\delta(a)^*=-\delta(a^*)$ for $a\in\B_1$.

As is well-known, there is typically no differential 
$\d:\,\Omega^1_T(\B_1,\B_2)\to \Omega^2_T(B,\B_2)$
such that $\pi_T\circ\delta=\d\circ\pi_T$. The naive formula
$$
\d\left(\sum_ib_i[T,c_i]\right)=\sum_i[T,b_i][T,c_i]
$$
is not well-defined. It can happen that 
$\sum_ib_i[T,c_i]=0$ while $\sum_i[T,b_i][T,c_i]\neq 0$.
The forms in $\J^2_T:=\pi_T(\delta(\ker(\pi_T)))$ are known as {\em junk forms}.

Below we will identify analytic spaces of represented one- and two-forms
which will serve as suitable receptacles for curvature. 

\subsubsection{The formal definitions of represented forms}

We again fix a $C^{2}$-Kasparov module $(\mathcal{B},Y,T)$.

By construction, the operator $*$-algebra $\mathcal{B}_{1}$ acts 
completely contractively on the Hilbert space $\Dom T$ equipped 
with the graph norm. Via this representation the $C^{*}$-algebra 
$\mathbb{B}(\Dom T)$ becomes an operator 
$\mathcal{B}_{1}$-bimodule. Furthermore recall that the 
operator norm on $\mathbb{B}(\Dom T)$ can be expressed as
\[
\|R\|_{\mathbb{B}(\Dom T)}=\|(T+i)R(T+i)^{-1}\|_{\mathbb{B}(Y)},\ \ \quad R\in \mathbb{B}(\Dom T),
\]
which we will exploit in the proof of the following Lemma.

\begin{lemma}
\label{indc2}
The map $b\mapsto [T,b]$ defines a completely bounded derivation 
$\delta_{T}:\mathcal{B}_{2}\to \mathbb{B}(\Dom T)$. 
Hence there is a completely bounded map
$$
\pi_{T}:\Omega^{1}_{u}(\mathcal{B}_{1},\mathcal{B}_{2})\to 
\mathbb{B}(\Dom T), \quad a\otimes b\mapsto a[T,b].
$$
The map $b\mapsto [T^{2},b]$ extends to a completely bounded derivation
\[
\delta_{T^{2}}:\mathcal{B}_{2}\to \mathbb{B}(\Dom T, Y),
\]
on $\mathcal{B}_{2}$. Hence there is a completely bounded map
\begin{align*}
\pi_{T^2}:\Omega^{1}_{u}(B,\mathcal{B}_{2})\to 
\mathbb{B}(\Dom T, Y),\quad  a\otimes b\mapsto a[T^{2},b].
\end{align*}
We can extend $\delta_T$ to a completely bounded  map 
\begin{align}
\delta_T:\,&\pi_T(\Omega^{1}_{u}(\mathcal{B}_{1},\mathcal{B}_{2}))
\to \BB(\Dom T,Y)\nonumber\\
 \delta_T(a[T,b])&:=T(a[T,b])+(a[T,b])T
=[T,a][T,b]+a[T^2,b]
\label{eq:delta-tee}
\end{align}
and
the composition 
$\delta_{T}\circ \pi_{T}:\Omega^{1}_{u}(\mathcal{B}_{1},\mathcal{B}_{2})\to \mathbb{B}(\Dom T, Y)$ 
is a  completely bounded map.
\end{lemma}
\begin{proof}The estimate
\begin{align*}
\Big\| \sum_{i}a_{i}[T,b_i]\Big\|_{\mathbb{B}(\Dom T)}
&=\Big\| (T+i)\Big(\sum_{i}a_{i}[T,b_i]\Big)(T+i)^{-1}\Big\|_{\mathbb{B}(Y)}\\
&\leq \Big\| \sum_{i}a_{i}[T,b_i]\Big\|_{\mathbb{B}(Y)}
+\Big\| \Big[T,\Big(\sum_{i}a_{i}[T,b_i]\Big)(T+i)^{-1}\Big]\Big\|_{\mathbb{B}(Y)}\\
&\leq \Big\| \sum_{i}a_{i}[T,b_i]\Big\|_{\mathbb{B}(Y)}
+\Big\| \sum_{i}[T,a_{i}][T,b_i]\Big\|_{\mathbb{B}(Y)}\\ 
&\quad\quad\quad\quad 
+\Big\|\sum_{i}a_{i}[T^{2},b_{i}](T+i)^{-1}\Big\|_{\mathbb{B}(Y)},
\end{align*}
shows that $\left\| \sum_{i}a_{i}[T,b_i]\right\|_{\mathbb{B}(\Dom T)}\leq 3\|\sum a_{i}\otimes b_{i}\|_{\mathcal{B}_{1}\otimes^{h}\mathcal{B}_{2}}$. 

For $b\in\mathcal{B}_{2}$, the operator $[T^{2},b]$ is defined on $\Dom T^{2}$ and 
$[T^{2},b](T+i)^{-1}:\Dom T\to Y$ extends to a bounded operator 
on $Y$. Thus if $\xi_{n}\in\Dom T^{2}\to \xi\in\Dom T$ in the 
graph norm of $T$, then $[T^{2},b]\xi_{n}=[T^{2},b](T+i)^{-1}(T+i)\xi_{n}$ 
is convergent. The estimate
 \[
 \|[T^{2},b]\xi\|\leq \|[T^{2},b](T+i)^{-1}\|\|(T+i)\xi\|,
 \]
holds on the $T$-graph-norm dense subspace 
$\Dom T^{2}\subset\Dom T$, and shows that the derivation $\delta_{T^2}$
is completely bounded as a map 
$\mathcal{B}_{2}\to \mathbb{B}(\Dom T, Y)$. We denote 
the $cb$-norm of $\delta_{T^2}$ by $\Vert\delta_{T^2}\Vert_{\textnormal{cb}}$.
The second statement now follows from the standard Haagerup estimate
\begin{align*}
\Big\|\sum_{i} a_{i}[T^{2},b_{i}]\Big\|^{2} 
&\leq \Big\|\sum a_{i}a_{i}^{*}\Big\|_{B} \Big\|\sum_{i} [T^{2},b_{i}]^{*}[T^{2},b_{i}]\Big\|_{\mathbb{B}(\Dom T, Y)} \\
&\leq\big\|\delta_{T^{2}}\big\|^{2}_{\textnormal{cb}} \big\|(a_{i})^{t}\big\|_{B}^{2} \big\| (b_{i})\big\|_{\mathcal{B}_{2}}^{2}.
\end{align*}
where $(a_i)^t$ is a row vector and $(b_i)$ a column.
Since $\pi_{T}:\mathcal{B}_{1}\otimes^{h}\mathcal{B}_{2}\to \mathbb{B}(Y)$, this proves that $\delta_{T}\circ \pi_T (\omega)$ is defined as an operator $\mathbb{B}(\Dom T, Y)$.
\end{proof}

The appearance of commutators $[T^2,b]$ for $b\in\B$ is a
consequence of the natural norm on the domain of $T$. Somewhat more
heuristically, it can be considered as a consequence of trying to extend
the derivation $b\mapsto [T,b]$ to one-forms via the graded commutator
$$
a[T,b]\mapsto \delta_T(a[T,b])=\big[T,a[T,b]\big]_+=[T,a][T,b]+a[T^2,b].
$$
Thus the graded commutator implements the ill-defined ``differential''
$a[T,b]\mapsto[T,a][T,b]$, up to the unwanted 
unbounded term $a[T^2,b]$, as in Equation \eqref{eq:delta-tee}.

\begin{defn} 
For a $C^{2}$-Kasparov module $(\B,Y,T)$ 
we denote by $\Omega^{1}_{T}(\mathcal{B}_{1},\mathcal{B}_{2})$ the closure of the image of the map
\[
\pi_{T}:\Omega^{1}_{u}(\mathcal{B}_{1},\mathcal{B}_{2})\to \mathbb{B}(\Dom T),
\]
and by $\Omega^{1}_{T}(B,\mathcal{B}_{1})$ the closure of the image of the map
\[
\pi_{T}:\Omega^{1}_u(B,\mathcal{B}_{1})\to \mathbb{B}(Y).
\]
\end{defn}

Note that $\Omega^{1}_{T}(B,\mathcal{B}_{1})$ contains the module 
of differential forms $\Omega^{1}_{T}$ defined by Connes in \cite[Chapter VI]{BRB} as a dense subspsace (for which only finite linear combinations of the $a[T,b]$ are allowed). However, in the case that $\Omega^1_{T}$ is a finitely generated projective (right) $\B$-module, they coincide.

\begin{defn}
\label{def: formsdomrep}
For a $C^{2}$-Kasparov module $(\B,Y,T)$  we denote by $\Omega^{2}_{T}(\mathcal{B}_{1})$ the closure of the image of the map $m\circ (\pi_{T}\otimes\pi_{T}),$ defined as the composition
\begin{equation}
\Omega^{1}_{u}(B,\mathcal{B}_{1})\otimes^{h}_{\mathcal{B}_{1}}\Omega^{1}_{u}(\mathcal{B}_{1},\mathcal{B}_{2})\to \Omega^{1}_{u}(B,\mathcal{B}_{1})\otimes^{h}_{\mathcal{B}_{1}}\Omega^{1}_{u}(B,\mathcal{B}_{1})\xrightarrow{\pi_{T}\otimes \pi_{T}} \mathbb{B}(Y)\otimes^{h}_{\mathcal{B}_{1}}\mathbb{B}(Y)\xrightarrow{m}\mathbb{B}(Y).\end{equation}
\end{defn}
\begin{lemma} 
\label{lem:junk-range}
The closure of the range of the map \[m\circ (\pi_{T}\otimes\pi_{T}):\Omega^{1}_{u}(B,\mathcal{B}_{1})\otimes^{h}_{\mathcal{B}_{1}}\Omega^{1}_{u}(\mathcal{B}_{1},\mathcal{B}_{2})\to \mathbb{B}(Y)\]
coincides with the space of two-forms
\[
\Omega^{2}_{T}(\mathcal{B}_{1})=\overline{\left\{\sum a_{i}[T,b_{i}][T,c_i]:\, a_i\in B,\ b_i,\,c_i\in\mathcal{B}_{1}\right\}},
\]
over $\mathcal{B}_{1}$.
\end{lemma}
\begin{proof} 
To see that the closure of the range coincides with 
$\Omega^{2}_{T}(\mathcal{B}_{1})$, consider a finite sum
\begin{align*}
\sum_{i,j}a_{i}[T,b_{i}]c_{j}[T,d_{j}]
=\sum_{i,j}a_{i}[T,b_{i}c_{j}][T,d_{j}]-a_{i}b_{i}[T,c_{j}][T,d_{j}]
= \sum_{i,j}a_{i}[T,b_{i}c_{j}][T,d_{j}],
\end{align*}
since $\sum_ia_ib_i=0$.
Hence we have an inclusion 
$\textnormal{ran } m\circ (\pi_{T}\otimes\pi_{T})\subset \Omega^{2}_{T}(\mathcal{B}_{1})$. 
Since any $c\in\mathcal{B}_{1}$ is a limit of $c_{i}\in\mathcal{B}_{2}$ such that $[T,c_i]\to[T,c]$, all expressions $a[T,b][T,c]$ are in the range of $m\circ (\pi_{T}\otimes\pi_{T})$ and
the reverse inclusion follows as well.
\end{proof}

\subsection{Second derivatives and junk}
\label{subsec:sec-deriv-junk}

We now provide a discussion of junk forms for Kasparov modules in our analytic context. The rather strange representation of forms given by
$\pi_{T^2}$ turns out to be precisely what is required to cancel 
out the unwanted term in the anticommutator $[T,a[T,b]]_+$. In turn,
this cancellation allows us to both 
represent curvature and capture junk.

\begin{prop}
\label{ajunkie} 
For any $\omega\in\Omega^{1}_{u}(\mathcal{B}_{1},\mathcal{B}_{2})$ 
we have that 
$$
\delta_{T}\circ \pi_{T}(\omega)-\pi_{T^2}(\omega)=m\circ (\pi_{T}\otimes\pi_{T})(\delta \omega)
\in\Omega^{2}_{T}(\mathcal{B}_{1})\subset \mathbb{B}(Y).
$$ 
The map $\delta_{T}\circ \pi_{T}-\pi_{T^2}:\Omega^{1}_{u}(\mathcal{B}_{1},\mathcal{B}_{2})\to \Omega^{2}_{T}(\mathcal{B}_{1})$ 
is completely contractive.
In particular for 
$\omega\in\Omega_{u}^{1}(\mathcal{B}_{1}, \mathcal{B}_{2})$ 
such that $\pi_{T}(\omega)=0$,  the operator 
$\pi_{T^2}(\omega)$ is well-defined and bounded on $Y$. 
\end{prop}
\begin{proof} 
Let $\alpha=(\alpha_{ij})\in M_{N}(\Omega^{1}_{u}(\mathcal{B}_{1},\mathcal{B}_{2}))$. Approximate $\alpha$ by a series 
\begin{equation}
\label{Hmagic}
\alpha_{n}:=\left(\sum_{k=1}^{n} a_{ik}\otimes b_{kj}\right)_{ij}\in  M_{N}\big((\mathcal{B}_{1}\otimes^{\textnormal{alg}}\mathcal{B}_{2})\cap\Omega^{1}_{u}(\mathcal{B}_{1},\mathcal{B}_{2})\big),
\end{equation} 
using Proposition \ref{Haaconvergent}.1. By abuse of notation we denote the self-adjoint regular operator $\diag T:(\Dom T)^{N}\to Y^{N}$ by $T$ and we identify $M_{N}(\mathbb{B}(Y))$ with $\mathbb{B}(Y^{N})$.
The identity
\[
\left[T, \pi_{T}(\alpha_{n})\right]=\left(\Big[T,\sum_{k=1}^{n} a_{ik}[T,b_{kj}]\Big]_+\right)_{ij}
=\left(\sum_{k=1}^{n} [T,a_{ik}][T,b_{kj}]+\sum_{k=1}^{n} a_{ik}[T^{2},b_{kj}]\right)_{ij},
\]
is valid in $\mathbb{B}(\Dom T, Y^{N})$. By continuity of 
$\delta_{T}\circ \pi_{T}$ the left hand side of this equation 
converges (as $n \to \infty$) in $\mathbb{B}(\Dom T, Y^{N})$. 
By continuity of $\pi_{T^2}$ we have 
$\pi_{T^2}(\alpha_{n})\to \pi_{T^2}(\alpha)$ in 
$\mathbb{B}(\Dom T, Y^{N})$. Therefore 
$\left(\sum_{k=1}^{n} [T,a_{ik}][T,b_{kj}]\right)_{ij}$ is convergent in 
$\mathbb{B}(\Dom T, Y^{N})$.
Since there is a complete contraction $\mathcal{B}_{2}\to \mathcal{B}_{1}$ 
we have the estimate
\[
\Big\|\Big(\sum_{k=\ell}^{m} [T,a_{ik}][T,b_{kj}]\Big)_{ij}\Big\|_{\mathbb{B}(Y^{N})}\!\!\!\leq\left\|(a_{ik})_{k= m}^{\ell}\right\|_{\mathcal{B}_{1}}\left\|(b_{kj})_{k= m}^{\ell}\right\|_{\mathcal{B}_{1}}\leq \left\|(a_{ik})_{k= m}^{\ell}\right\|_{\mathcal{B}_{1}}\left\|(b_{kj})_{k= m}^{\ell}\right\|_{\mathcal{B}_{2}} ,
\]
from which we infer that
\[
\left\|\left(\sum_{k=\ell}^{m} [T,a_{ik}][T,b_{kj}]\right)_{ij}\right\|_{\mathbb{B}(Y^{N})}
\leq \|\alpha_{k}-\alpha_{m}\|_{M_{N}(\Omega^{1}_{u}(\mathcal{B}_{1},\mathcal{B}_{2}))}\to 0.
\]
It thus follows that the series
$$
\left(\sum_{k} [T,a_{ik}][T,b_{kj}]\right)_{ij},
$$
is norm convergent in $\mathbb{B}(Y^{N})$, 
proving that 
$\delta_{T}\circ \pi_{T}(\omega)-\pi_{T^2}(\omega)\in M_{N}(\Omega^{2}_{T})
\subset \mathbb{B}(Y^{N})$. 
The above estimates show that
\[
\|(\delta_{T}\circ \pi_{T}-\pi_{T^2})(\omega)\|_{\mathbb{B}(Y^N)}
\leq \|\omega\|_{M_{N}(\Omega^{1}_{u}(\mathcal{B}_{1},\mathcal{B}_{2}))}, \quad
\omega\in M_{N}(\Omega^{1}_{u}(\mathcal{B}_{1},\mathcal{B}_{2})),
\]
which proves that $\delta_{T}\circ \pi_{T}-\pi_{T^2}$ is completely contractive.
In case $\omega\in\Omega^{1}_{u}(\mathcal{B}_{1},\mathcal{B}_{2})$ and $\pi_{T}(\omega)=\sum_ia_i[T,b_i]=0$ we find that
\begin{align*}
\pi_{T^2}(\omega)&=\sum_ia_i[T^2,b_i]=\sum_ia_iT[T,b_i]+a_i[T,b_i]T\\
&=\sum_i[a_i,T][T,b_i]+\sum_iTa_i[T,b_i]=-\sum_{i}[T,a_{i}][T,b_{i}]\in \mathbb{B}(\Dom T, Y).
\end{align*}
This proves that $\pi_{T^2}(\omega)$ extends to  a bounded operator on $Y$. 
\end{proof}
\begin{corl} 
\label{cor:junk-squared}
Let $\pi_{T^2}:\Omega^{1}_{u}(\mathcal{B}_{1},\mathcal{B}_{2})
\to \Omega^{1}_{T^{2}}(\mathcal{B}_{1},\mathcal{B}_{2})$ and 
$\pi_{T}:\Omega^{1}_{u}(\mathcal{B}_{1},\mathcal{B}_{2})
\to \Omega^{1}_{T}(\mathcal{B}_{1},\mathcal{B}_{2})$ be the universal maps. 
Then $\ker \pi_{T}$ is a closed subbimodule of 
$\Omega^{1}_{u}(\mathcal{B}_{1},\mathcal{B}_{2})$ 
and the closed bimodule of junk forms 
\[
\mathcal{J}^{2}_{T}(\B_1)
:=\overline{\left\{\sum_{i}[T,a_i][T,b_i]: \sum a_{i}b_{i}
=\sum a_i[T,b_i]=\sum[T,a_i]b_i=0\right\}},
\]
is equal to the closure of the space $\pi_{T^2}(\ker \pi_{T})$.
\end{corl}

\section{Curvature in unbounded KK-theory}
\label{Hilbcurve}

We now come to the main construction in this paper, which is the notion of curvature in the context of the unbounded Kasparov product. As mentioned in the introduction, it is our goal to make sense of the following formula 
\begin{equation*}
R_{(S,\nabla_{T})}=(S\otimes 1+1\otimes_{\nabla}T)^{2}-S^{2}\otimes 1
-1\otimes_{\nabla}T^{2}.
\end{equation*}
For this we again fix an unbounded Kasparov module
$(\mathcal{B},Y,T)$ which will provide our reference `horizontal' differential structure.
For the most part we make no use of the (locally) compact resolvent
of $T$, only occasionally (and always explicitly) 
requiring that $(\mathcal{B},Y,T)$
defines a $KK$-class. More important are the various modules of forms, 
junk and representations
$$
\Omega^1_T(\B_*,\B_*),\quad \mathcal{J}_T,\quad  \pi_T,\quad \pi_{T^2}
$$
defined as in the last section.

In order to define the curvature we need to introduce a suitable notion of $C^{2}$-connection, along with some `vertical' differentiability conditions on the $C^*$-module $X$. As we will see, the latter is phrased naturally in terms of the self-adjoint regular operator $S$ on $X$.

\subsection{The definition of $C^{1}$ and $C^{2}$-connections}
In order to define curvature we need a suitable notion of $C^{2}$-connection.
We require the notion of ``form-valued inner products''. If
we have $x,\,y\in X$ and $\omega\in \Omega^1_u(\B_{1},\B_2)$ (for instance),
we define pairings
\begin{equation}
\langle x, y\ox\omega\rangle:=\langle x,y\rangle\ox\omega,
\quad\mbox{and}\quad
\langle x\ox\omega, y\rangle:=\omega^*\ox\langle x,y\rangle,
\label{eq:do-we-need-to-refer-to-this}
\end{equation}
which are compatible with the balancing relation over $\B_{1}$.
\begin{defn} 
\label{diffconn}
Let $X$ be a $\Z_2$-graded $C^{*}$-module over $B$ and $(\mathcal{B},Y,T)$ a $C^{k}$- unbounded Kasparov module ($k=1,2$). A dense $\mathcal{B}_{k}$-submodule $\mathcal{X}\subset X$ is \emph{horizontally $C^{k}$-differentiable} with respect to $(\mathcal{B},Y, T)$ if for all $x,y\in\mathcal{X}$ it holds that $\langle x,y\rangle \in \mathcal{B}_{k}$. 
\end{defn}
\begin{rmk}
We will develop curvature with the minimal smoothness assumptions that we can, working with horizontally differentiable  $C^{1}$-modules where possible, and imposing further $C^{2}$-structure as we need it.
\end{rmk}
\begin{defn}
\label{defn:connection}
A  universal \emph{connection} on $\mathcal{X}$ is a linear map, so if $\gamma$ is the grading operator of $X$ then
\[
\nabla:\mathcal{X}\to X\otimes^{h}_{B}\Omega^{1}_{u}(B,\mathcal{B}_{1}),
\quad 
\nabla(\gamma(x))=(\gamma\otimes 1)\nabla(x),
\]
which satisfies the Leibniz rule
\[\nabla(xb)=\nabla(x)b+\gamma(x)\otimes\delta(b),\quad \forall x\in\mathcal{X} \quad b\in\mathcal{B}_{1}.\]
If in addition $\nabla$ satisfies 
\[\langle \gamma(x),\nabla(y)\rangle -\langle \nabla(\gamma(x)),y\rangle=\delta (\langle x,y\rangle),\quad \forall x,y\in\mathcal{X},\]
then $\nabla$ is said to be \emph{Hermitian} or compatible. 
%
\end{defn}
We write $\nabla_{T}:=\pi_{T}\circ\nabla:\mathcal{X}\to X\otimes^{h}_{B}\Omega^{1}_{T}(\mathcal{B}_{1})$ and call the composition the
\emph{represented connection induced by} $\nabla$.
\begin{lemma}
\label{symmetric}
Let $\mathcal{X}$ be a horizontally differentiable $C^{1}$-module for the unbounded Kasparov module $(\mathcal{B},Y,T)$  and 
$\nabla:\mathcal{X}\to X\otimes^{h}_{B}\Omega^{1}_{u}(B,\mathcal{B}_{1})$
a Hermitian connection. Then
\[
1\otimes_{\nabla}T:\mathcal{X}\otimes^{\alg}_{\mathcal{B}}\Dom T
\to X\otimes_{B}Y,\quad x\otimes y\mapsto \gamma(x)\otimes Ty+\nabla_{T}(x)y,
\]
is a densely-defined odd symmetric operator in $X\otimes_{B}Y$. 
\end{lemma}

\begin{proof} On $X\otimes_{B}\Omega^{1}_{u}(B,\mathcal{B}_{1})$ it holds that $(1\otimes \gamma)(x\otimes\omega_{T})y=-x\otimes \omega_{T}\gamma(y)$. Since $\gamma\otimes \gamma=(\gamma\otimes 1)(1\otimes \gamma),$ for 
$x\otimes y\in \mathcal{X}\otimes^{\alg}_{\mathcal{B}}\Dom T$ we have
\begin{align*}
(\gamma\otimes \gamma)(1\otimes_{\nabla}T)(x\otimes y)&=(\gamma\otimes\gamma)(\gamma(x)\otimes Ty +\nabla_{T}(x)y)\\
&=-x\otimes T\gamma(y)-(\gamma\otimes 1)\nabla_{T}(x)\gamma(y)\\
&=-x\otimes T\gamma(y)-\nabla_{T}(\gamma(x))\gamma(y)\\
&=-(1\otimes_{\nabla}T)(\gamma\otimes\gamma)(x\otimes y),
\end{align*}
so $1\otimes_{\nabla}T$ is an odd operator. For symmetry, we write
\begin{align*}
\big\langle (1\otimes_{\nabla}T)&(x\otimes y), x\otimes y\big\rangle 
=\big\langle \langle x,\gamma(x)\rangle Ty,  y\rangle+\langle \nabla_{T}(x)y, x\otimes y\big\rangle\\
&=-\big\langle \left[T,\langle \gamma(x),x\rangle\right]y,  y\rangle+\langle \langle \gamma(x),x\rangle y, Ty\rangle+\langle \nabla_{T}(x)y, x\otimes y\big\rangle\\
&=\big\langle\langle \nabla_{T}(x),x\rangle  y, y\rangle-\langle\langle x,\nabla_{T}(x)\rangle  y, y\rangle+\langle y, \langle x,\gamma(x)\rangle Ty\rangle+\langle \nabla_{T}(x)y, x\otimes y\big\rangle\\
&=\big\langle y, \langle x,\nabla_{T}(x)\rangle y\rangle+\langle y, \langle x,\gamma(x)\rangle Ty\big\rangle\\
&=\big\langle x\otimes y, \gamma(x)\otimes Ty\rangle+\langle x\otimes y, \nabla_{T}(x)y\big\rangle\\
&=\big\langle x\otimes y, (1\otimes_{\nabla}T)(x\otimes y)\big\rangle. 
\end{align*}
Polarisation completes the proof.
\end{proof}


A horizontal $C^1$-structure is all one requires for finitely generated modules, but it is not sufficient to describe the connections and curvature of countably generated modules. In order to define appropriate notions of connections, we need to introduce some further smoothness on the $C^{*}$-module $X$. 

Let $S:\Dom S\to X$ 
be a self-adjoint regular operator on $X$. 
We think of $S$ as defining a vertical 
differential structure on $X$. The presence of a compatible horizontal differentiable structure for a Kasparov 
module $(\mathcal{B},Y,T)$ then provides us with the correct notion of differentiable submodule.

We do not require that the data $(X,S)$ define an unbounded Kasparov module, although this is often the case in examples.

\begin{defn}
\label{def: horver}
Let $(\mathcal{B},Y,T)$ be a $C^{1}$-Kasparov module. 
A $C^{1}$-\emph{module} for $(\mathcal{B},Y,T)$ 
is a pair $(\mathcal{X},S)$ such that
\begin{enumerate}
\item $\mathcal{X}$ is a horizontally differentiable $C^{1}$-module with $C^{*}$-closure $X$;
\item $S:\mathcal{X}\to X$ is an essentially self-adjoint and regular operator on $X$.
\end{enumerate}
A $C^{1}$-\emph{connection} on a $C^{1}$-module $(\mathcal{X},S)$ is a Hermitian connection 
$$
\nabla:\mathcal{X}\to X\otimes_{B}\Omega^{1}_{u}(B,\mathcal{B}_{1}),
$$ 
such that 
$S\otimes 1+1\otimes_{\nabla}T:\mathcal{X}\otimes^{\alg}_{\mathcal{B}_{1}}\Dom T\to X\otimes_{B} Y$ is essentially self-adjoint and regular.
\end{defn}
\begin{rmk} 
The choice $S=0$ is allowed for defining a vertical differential structure.
\end{rmk}


Given a $C^{1}$-Hermitian connection
$
\nabla:\mathcal{X}\to X\otimes_{B}^{h}\Omega^{1}_{u}(B,\mathcal{B}_{1})
$
we can define a completion of $\mathcal{X}$ by
\begin{equation}
\mathcal{X}_{\nabla_{T}}:=\big\{x\in X:\exists x_{n}\in\mathcal{X},\ x_n\to x\in X,
\quad \nabla_{T}(x_n)\to \nabla_{T}(x)\in X\otimes_{B}^{h}\Omega^{1}_{T}(\mathcal{B}_{1})\big\}.
\label{eq:ex-nabla-tee}
\end{equation}
\begin{prop}[\cite{BKM} Section 3.6, Lemma 3.4]
\label{prop: nablaclosure}
Let $(\mathcal{B},Y,T)$ be a $C^{1}$-Kasparov module and
$(\mathcal{X},S)$ a $C^{1}$-module.
The space $\mathcal{X}_{\nabla_{T}}$ is an operator $*$-module 
over $\mathcal{B}_{1}$. Moreover, for every $x\in\mathcal{X}_{\nabla_{T}}$
the linear map $|x\rangle:\,Y\to X\ox_BY$, $y\mapsto x\ox y$ satisfies
\[
|x\rangle :\Dom T\to \Dom 1\otimes_{\nabla}T,\quad \mbox{and}\quad
\big(|\gamma(x)\rangle T-1\otimes_{\nabla}T|x\rangle\big)\in 
\mathbb{B}(Y,X\otimes_{B}Y). 
\]
\end{prop}
We define $X^{S}$ to the completion of $X$ in the norm 
\[
\|x\|^{S}:=\|(S+i)^{-1}x\|=\|(S-i)^{-1}x\|.
\] 
The norm on $X^{S}$ is induced from the inner product 
\[
\langle x,y\rangle^{S}:=\langle (S+i)^{-1}x, (S+i)^{-1}y\rangle,
\]
and $X^{S}$ is a Hilbert $C^{*}$-module in this inner product. 
It is appropriate to think of $X^{S}$ as a degree $-1$ 
Sobolev space associated to $S$, as the following observation shows.
\begin{lemma} 
\label{Soboweird}
For $x\in X$ the maps $x\mapsto (S+i)^{-1}x,$ $x\mapsto (S-i)^{-1}x$ can be viewed as densely defined
maps $X^{S}\to X$ and these 
extend to unitary isomorphisms $X^{S}\to X$.
\end{lemma}

\begin{rmk}
\label{rmk:sobo-rmk}
One naturally expects that defining second derivatives 
would require two
Sobolev spaces, normally $W^{2}_{2}\to W^{2}_{1}\to L^2$. In order
to accommodate countably generated modules and Kasparov 
modules with only a $C^{1}$-structure, we take a slightly different route and introduce a space $X_{\nabla_T}^S$ that behaves like a $-1$ Sobolev space in the vertical direction but a $+1$ Sobolev in the horizontal direction. 


\end{rmk}


We define $\mathcal{X}^{S}_{\nabla_{T}}$ to be the completion of 
$\mathcal{X}$ in the operator space topology induced by the norm
\begin{equation}
\|x\|_{\nabla_{T}}^{S}:=\max\left\{\|(S+i)^{-1}x\|_{X},\|
(S+i)^{-1}\nabla_{T} (x)\|_{X\otimes^{h}_{B}\Omega^{1}_{T}(B,\B_1)}\right\}.
\label{eq:ex-nabla-tee-ess}
\end{equation}
Since the norm on $\mathcal{X}_{\nabla_{T}}$ dominates the norm on $\mathcal{X}_{\nabla_{T}}^{S}$, 
the identity map on $\mathcal{X}$ extends to a complete contraction $\iota_{S}:\mathcal{X}_{\nabla_{T}}\to \mathcal{X}_{\nabla_{T}}^{S}$. 
To define the appropriate notion of $C^{2}$-connection, we need the operators $S\otimes 1$ and $1\otimes_{\nabla}T$ to
be compatible in a more precise way. We introduce the following definition.
\begin{defn}[cf. \cite{KaLe, LM, MR}]
\label{wac}
Let $(s,t)$ be self-adjoint regular operators in 
the Hilbert $C^{*}$-module $E$. We say that 
$(s,t)$ is a \emph{vertically anticommuting pair} if
\begin{enumerate}
\item $(s\pm i)^{-1}:\Dom t \to \mathcal{F}(s,t)
:=\{e\in\Dom s\cap \Dom t:\, se\in\Dom t,\  te\in\Dom s\}$;
\item $st+ts:\mathcal{F}(s,t)\to E$ extends to $\Dom s$.
\end{enumerate}
\end{defn}
\begin{rmk}
The definition of vertically anticommuting pair is an 
asymmetric version of \cite[Definition 2.1]{LM} 
and was used in \cite{KaLe, MR}. The main result is that 
$s+t$ is self-adjoint and regular on $\Dom s\cap \Dom t$.
In this paper we require the more 
restrictive asymmetric version, which is sufficient to cover many geometric examples, is 
compatible with the unbounded Kasparov product and seems to
be necessary for technical reasons.
\end{rmk}


\begin{lemma}
\label{lem: XSinclusion}
Let $(\mathcal{X},S)$ be a $C^{1}$-module over the Kasparov module $(\mathcal{B},Y, T)$ and  $\nabla:\mathcal{X}\to X\otimes^{h}_{B}\Omega^{1}_{T}(B,\mathcal{B}_{1})$ a connection such that $(S\otimes 1,1\otimes_{\nabla}T)$ is a vertically anticommuting pair.
Then the identity map extends to a completely contractive injection $\mathcal{X}^{S}_{\nabla_{T}}\to X^{S}$.
\end{lemma}
\begin{proof}
Since the norm on $\X^S_{\nabla_{T}}$ dominates the norm on $X^{S}$, 
the fact that the identity map on $\mathcal{X}$ extends to a complete contraction is immediate. To see that this map is injective,
let $x_{n}\in\mathcal{X}$ be a sequence such that $(S+i)^{-1}x_{n}\to 0$ and $(S+i)^{-1}\nabla_{T}(x_{n})$ is convergent. For
$x\in\mathcal{X}$ and $y\in\Dom T$  it holds that 
\[\nabla_{T}(x)y=(1\otimes_{\nabla}T)(x\otimes y)-\gamma(x)\otimes Ty.\]
Using this and the fact that $(S\otimes 1,1\otimes_{\nabla}T)$ is a vertically anticommuting pair we write
\begin{align*}
(S+i)^{-1}\nabla_{T}(x_n)y &=(S+i)^{-1}(1\otimes_{\nabla}T)(x_n\otimes y)-(S+i)^{-1}\gamma(x_n)\otimes Ty\\
&=-(1\otimes_{\nabla}T)(S-i)^{-1}(x_n)\otimes y+\gamma ((S-i)^{-1}x_n)\otimes Ty\\
&\quad\quad\quad + (S+i)^{-1}[S\otimes 1,1\otimes_{\nabla}T](S-i)^{-1}x_n\otimes y.
\end{align*}
Since $(S-i)^{-1}x_n$ is convergent to $0$ and $(S+i)^{-1}[S\otimes 1,1\otimes_{\nabla}T]$ is bounded, it follows that $(1\otimes_{\nabla}T)(S-i)^{-1}(x_n)\otimes y$ is convergent. Since $1\otimes_{\nabla}T$ is closed, the limit of the latter sequence must be $0$. Thus $(S+i)^{-1}\nabla_{T}(x_n)y $ converges to $0$ for $y\in\Dom T$. Since $(S+i)^{-1}\nabla_{T}(x_n)$ is convergent in $\mathbb{B}(Y, X\otimes_{B}Y)$, its limit must be $0$.
\end{proof}
Before introducing the curvature operator we require some 
technical domain results.

\begin{lemma} 
\label{domaineq}
Let $(\mathcal{X},S)$ be a $C^{1}$-module over the Kasparov module $(\mathcal{B},Y, T)$ and  let $\nabla~:\mathcal{X}\to X\otimes^{h}_{B}\Omega^{1}_{T}(B,\mathcal{B}_{1})$ be a connection such that $(S\otimes 1,1\otimes_{\nabla}T)$ is a vertically anticommuting pair. Then: 
\begin{enumerate}
\item 
$(1\otimes_{\nabla}T) (S+i)^{-1}\ox1$ is defined on 
$\Dom 1\otimes_{\nabla}T$ and closable. There is an equality of domains
\begin{align*}
\Dom\big( (1\otimes_{\nabla}T)(S\pm i)^{-1}\ox1\big)
&=\Dom\big( (S\mp i)^{-1}\ox1(1\otimes_{\nabla}T)\big),
\end{align*}
of closed operators in $X\otimes_{B}Y$.
\item $\nabla:\mathcal{X}\to X\otimes_{B}\Omega^{1}_{T}(B,\mathcal{B}_{1})$ extends to a map $\nabla:\mathcal{X}^{S}_{\nabla_{T}}\to X^{S}\otimes_{B}\Omega^{1}_{T}(B,\mathcal{B}_{1})$.
\item $(S+i)^{-1}\ox1:X^{S}\otimes_{B}Y\to X\otimes_{B}Y$ restricts to a map $(S+i)^{-1}\ox1:\mathcal{X}^{S}_{\nabla_{T}}\otimes^{h}_{\mathcal{B}_{1}}\Dom T\to\Dom (1\otimes_{\nabla}T)$.
\end{enumerate}
\end{lemma}
\begin{proof}
In the following we will frequently write $S$ for $S\ox 1$.
We first prove that 
\[
\Dom\big( (1\otimes_{\nabla}T)(S+i)^{-1}\big)
=\Dom\big( (S-i)^{-1}(1\otimes_{\nabla}T)\big).
\] 
Since $(S\otimes 1,1\otimes_{\nabla}T)$ is a vertically
anticommuting pair, both operators are initially defined on $\Dom 1\otimes_{\nabla}T$ 
and 
\[
(S-i)^{-1}(1\otimes_{\nabla}T)+(1\otimes_{\nabla}T)(S+i)^{-1}
=(S-i)^{-1}[1\otimes_{\nabla}T,S](S+i)^{-1},
\]
is a bounded operator. Hence if $\xi_{n}\to \xi$ in 
$X\otimes_{B}Y$ then $(1\otimes_{\nabla}T)(S+i)^{-1}\xi_{n}$ 
converges if and only if $(S-i)^{-1}(1\otimes_{\nabla}T)\xi_{n}$ 
converges. So both operators are closable and have the same closure. 

2. Suppose that $x_{n}\in\mathcal{X}$ is a sequence converging to $x\in\mathcal{X}_{\nabla_{T}}^{S}$, that is $(S+i)^{-1}x_{n}$ and $(S+i)^{-1}\nabla_{T}(x_n)$ are both Cauchy sequences. Thus $\nabla_{T}(x_{n})$ converges to an element $z\in X^{S}\otimes^{h}_{B}\Omega^{1}_{T}(B,\mathcal{B}_{1})$. By Lemma \ref{lem: XSinclusion}, if $x_n\to 0\in\mathcal{X}_{\nabla_{T}}^{S}$ then $z=0$ and the map $x\mapsto \nabla_{T}(x):=z\in X^{S}\otimes_{B}\Omega^{1}_{T}(B,\mathcal{B}_{1})$ is well defined for $x\in \mathcal{X}_{\nabla_{T}}^{S}$.

3. We use that $(S\otimes 1,1\otimes_{\nabla}T)$ is a vertically anticommuting pair. Suppose that $x\in \mathcal{X}^{S}_{\nabla_{T}}$, which by construction is a submodule of $X^{S}$, so that $(S+i)^{-1}x\in X$. 

Take a finite row $\xi=(x_{1},\dots, x_{n})$ with $x_{i}\in\mathcal{X}_{\nabla_{T}}$ and a finite column $\eta=(y_{1},\dots, y_{n})^{t}$ with $y_{j}\in\Dom T$. Observe that 
\[
(S+i)^{-1}x_{i}\otimes y_{i}\in \mathcal{F}(S\otimes 1, 1\otimes_{\nabla}T)\subset  \Dom 1\otimes_{\nabla}T,\quad i=1,\dots, n,
\]
and write $\xi\otimes\eta=\sum_{i=1}^{n} x_{i}\otimes y_{i}\in  \mathcal{X}_{\nabla_{T}}^{S}\otimes^{\alg}_{\mathcal{B}_{1}}\Dom T$.
Now consider
\begin{align*}
1\otimes_{\nabla}T(S+i)^{-1}\xi\otimes \eta &=(S-i)^{-1}[S\otimes 1,1\otimes_{\nabla}T](S+i)^{-1}\xi\otimes \eta \\
&\quad\quad\quad\quad\quad\quad\quad\quad -(S-i)^{-1}(1\otimes_{\nabla}T)(\xi\otimes \eta)\\
&=(S-i)^{-1}[S\otimes 1,1\otimes_{\nabla}T](S+i)^{-1}\xi\otimes \eta\\
&\quad\quad\quad\quad\quad\quad\quad\quad -(S-i)^{-1}\gamma(\xi)\otimes T\eta-(S-i)^{-1}\nabla_{T}(\xi)\eta\\
&=(S-i)^{-1}[S\otimes 1,1\otimes_{\nabla}T](S+i)^{-1}\xi\otimes \eta\\
&\quad\quad\quad\quad\quad\quad\quad\quad +\gamma((S+i)^{-1}\xi)\otimes T\eta-(S-i)^{-1}\nabla_{T}(\xi)\eta,
\end{align*}
from which, using Theorem \ref{thm: Blecher}, we obtain the estimate
\begin{align}
\left\|1\otimes_{\nabla}T(S+i)^{-1}\xi\otimes \eta \right\|_{X\otimes_{B}Y}\leq C\|(S+&i)^{-1}\xi\|_{X}\|\eta\|_{Y} +\|(S+i)^{-1}\xi\|_{X}\|T\eta\|_{Y}\nonumber\\
&+\|(S-i)^{-1}\nabla_{T}(\xi)\|_{X\otimes^{h}_{B}\Omega^{1}_{T}(B,\mathcal{B}_{1})}\|\eta\|_{Y}.
\label{eq:this-is-this-thisly}
\end{align}
The estimate \eqref{eq:this-is-this-thisly} implies that
\begin{align*}
\left\|1\otimes_{\nabla}T(S+i)^{-1}\xi\otimes \eta \right\|_{X\otimes_{B}Y}\leq C\|\xi\|_{\mathcal{X}^{S}_{\nabla_{T}}}\|\eta\|_{\Dom T},
\end{align*}
and thus that 
\begin{align*}
\left\|1\otimes_{\nabla}T(S+i)^{-1}\xi\otimes \eta \right\|\leq C\|\xi\otimes \eta\|_{\mathcal{X}^{S}_{\nabla_{T}}\otimes^{h}_{\mathcal{B}_{1}}\Dom T}.
\end{align*}

Since $\mathcal{X}_{{\nabla}_{T}}\otimes^{\alg}_{\mathcal{B}_{1}}\Dom T$ is dense in $\mathcal{X}^{S}_{\nabla_{T}}\otimes^{h}_{\mathcal{B}_{1}}\Dom T$, the result follows.\end{proof}
Suppose now that we are given a map
\[
\nabla^{S}:\mathcal{X}\to \mathcal{X}^{S}_{\nabla_{T}}
\otimes^{h}_{\mathcal{B}_{1}}\Omega^{1}_{u}(\mathcal{B}_{1},\mathcal{B}_{2}),
\]
satisfying the Leibniz rule 
$\nabla^{S}(xb)=\nabla^{S}(x)b+\gamma(x)\otimes \delta (b)$ 
for all $b\in\mathcal{B}_{2}$. Then we obtain a well-defined operator
\[
1\otimes_{\nabla^{S}}T:\mathcal{X}\otimes^{\alg}_{\mathcal{B}_{1}}\Dom T 
\to \mathcal{X}^{S}_{\nabla_{T}}\otimes^{h}_{\mathcal{B}_{1}}Y
\to  X^{S}\otimes_{B}Y,\quad x\otimes y\mapsto \gamma(x)\otimes Ty
+\nabla_{T}^{S}(x)y.
\]
By composing $1\otimes_{\nabla^{S}}T$ 
with the resolvent $(S+i)^{-1}:X^{S}\to X$, 
which is defined on all of $X^{S}$ by Lemma \ref{Soboweird}, 
we obtain a well-defined map
\[
(S+i)^{-1}\cdot 1\otimes_{\nabla^{S}}T:
\mathcal{X}\otimes^{\alg}_{\mathcal{B}_{1}}\Dom T\to X\otimes_{B} Y.
\]
\begin{defn} 
\label{def: connpair}
Let $(\mathcal{X},S)$ be a $C^{1}$-module over the $C^1$-Kasparov module $(\mathcal{B},Y,T)$. 
By a \emph{$C^{2}$-connection} on $\mathcal{X}$ we mean a pair 
$(\nabla,\nabla^{S})$ of connections
\[
\nabla:\mathcal{X}\to X\otimes^{h}_{B}\Omega^{1}_{u}(B,\mathcal{B}_{1}),
\quad \nabla^{S}:\mathcal{X}\to \mathcal{X}^{S}_{\nabla_{T}}
\otimes^{h}_{\mathcal{B}_{1}}
\Omega^{1}_{u}(\mathcal{B}_{1}),
\]
with $\nabla$ a Hermitian connection and $\nabla^{S}$ a connection, 
such that
\begin{enumerate}
\item $(S\otimes 1, 1\otimes_{\nabla}T)$ is a vertically anticommuting pair;
\item for all $x\in\mathcal{X}$ and $y\in\Dom T$ we have
\[
(S+i)^{-1}\nabla_{T}(x)y
=(S+i)^{-1}\cdot \nabla^{S}_{T}(x)y \in X\otimes_{B}Y.
\]
\end{enumerate}
\end{defn}
Note that this definition implies that for all $x\in\mathcal{X}$ and $y\in\Dom T$ it holds that 
$(S+i)^{-1}\cdot (1\otimes_{\nabla^{S}}T)(x\otimes y)\in\Dom S$ 
and thus that 
$1\otimes_{\nabla^{S}}T(x\otimes y)=1\otimes_{\nabla}T(x\otimes y)$ 
is in fact an element of $X\otimes_{B}Y$ viewed as a subspace of 
$X^{S}\otimes_{B}Y$ via the dense inclusion $X\to X^{S}$.

\subsection{The represented curvature of a $C^{2}$-connection on $C^{1}$-module}

Now that we have a clear picture of how represented $C^{2}$-connections determine the domains of induced operators, we can 
introduce the 
represented curvature of a $C^{2}$-connection on a $C^{1}$-module $(\mathcal{X},S)$. To make appropriate sense of $\pi_{T}(\nabla^{2})$ a little care is required and this operator will more correctly be written $\pi_{T}((1\otimes_{\nabla_{T}}\delta)\circ\nabla^{S}),$ which yields a well-defined operator. 

\begin{prop} 
\label{nablassquare}
Let $(\nabla,\nabla^{S})$ be a $C^{2}$-connection on a $C^{1}$-module $(\mathcal{X},S)$ over a $C^{1}$-Kasparov module $(\mathcal{B},Y,T)$. The map
\begin{align*}
1\otimes_{\nabla_{T}}\delta:\mathcal{X}^{S}_{\nabla}\otimes^{h}_{\mathcal{B}_{1}}\Omega^{1}_{u}(\mathcal{B}_{1})&\to X^{S}\otimes_{B}^{h}\Omega^{1}_{T}(B,\mathcal{B}_{1})\otimes^{h}_{\mathcal{B}_{1}}\Omega^{1}_{u}(\mathcal{B}_{1}),\\ x\otimes\omega &\mapsto \nabla_{T}(x)\otimes \omega+\gamma(x)\otimes(\pi_{T}\otimes 1)(\delta\omega),
\end{align*}
is well-defined and satisfies $(1\otimes_{\nabla_{T}}\delta)(\eta b)=(1\otimes_{\nabla_{T}}\delta)(\eta)b-(\gamma\otimes\pi_{T})(\eta)\otimes\delta(b)$, for all $\eta\in \mathcal{X}^{S}_{\nabla}\otimes^{h}_{\mathcal{B}_{1}}\Omega^{1}_{u}(\mathcal{B}_{1})$ and $b\in \mathcal{B}_{1}$.
\end{prop}
\begin{proof} 
The map $\nabla_{T}:\mathcal{X}
\to X\otimes^{h}_{B}\Omega^{1}_{T}(B, \mathcal{B}_{1})$ extends to a map 
$$
\nabla_{T}:\mathcal{X}_{\nabla_{T}}^{S}
\to X^{S}\otimes_{B}^{h}\Omega^{1}_{T}(B,\mathcal{B}_{1}),
$$
by part 2 of Lemma \ref{domaineq}.
The maps
$$\delta:\Omega^{1}_{u}(\mathcal{B}_{1})
\to \Omega^{1}_{u}(B,\mathcal{B}_{1})\otimes_{\mathcal{B}_{1}}^{h}\Omega^{1}_{u}(\mathcal{B}_{1})$$ and 
\begin{align*}\pi_{T}\otimes 1:\Omega^{1}_{u}(B,\mathcal{B}_{1})\otimes_{\mathcal{B}_{1}}^{h}\Omega^{1}_{u}(\mathcal{B}_{1})\to \Omega^{1}_{T}(B,\mathcal{B}_{1})\otimes^{h}_{\mathcal{B}_{1}}\Omega^{1}_{u}(\mathcal{B}_{1})\end{align*} are completely bounded, and therefore
\begin{align}
(1\otimes \pi_{T}\otimes 1)\circ(\nabla\otimes 1+ \gamma\otimes \delta):\mathcal{X}^{S}_{\nabla}\otimes^{h}_{\mathbb{C}}\Omega^{1}_{u}(\mathcal{B}_{1})&\to X^{S}\otimes_{B}^{h}\Omega^{1}_{T}(B,\mathcal{B}_{1})\otimes^{h}_{\mathcal{B}_{1}}\Omega^{1}_{u}(\mathcal{B}_{1})\nonumber\\
x\otimes\omega &\mapsto \nabla_{T}(x)\otimes \omega+\gamma(x)\otimes(\pi_{T}\otimes 1)(\delta\omega),
\label{eq:the-map}
\end{align}
is well-defined and completely bounded. Since also 
\begin{align*}
\nabla_{T}(xb)\otimes \omega+\gamma(x)\otimes(\pi_{T}\otimes 1)( b\delta\omega)
&=\nabla_{T}(x)b\omega+\gamma(x)\otimes(\delta_{T} b)\otimes\omega\\ &\quad\quad\quad\quad\quad\quad +\gamma(x)\otimes (\pi_{T}\otimes 1)(b\delta\omega)\\
&=\nabla_{T}(x)b\omega+\gamma(x)\otimes(\pi_{T}\otimes 1)( \delta (b\omega)),
\end{align*}
the map \eqref{eq:the-map} is compatible with the balancing relation and descends to a completely bounded map
\[
1\otimes_{\nabla_{T}}\delta:\mathcal{X}^{S}_{\nabla_{T}}\otimes^{h}_{\mathcal{B}_{1}}\Omega^{1}_{u}(\mathcal{B}_{1},\mathcal{B}_{k})
\to X^{S}\otimes_{B}^{h}\Omega^{1}_{T}(B,\mathcal{B}_{1})\otimes^{h}_{\mathcal{B}_{1}}\Omega^{1}_{u}(\mathcal{B}_{1}),
\]
which is the desired statement. The (graded) Leibniz rule for $1\otimes_{\nabla_{T}}\delta$ follows directly from the defining formula
and the graded Leibniz rule for $\delta$, $\delta(\omega b)=\delta(\omega)b-\omega\otimes \delta(b)$, for $\omega\in\Omega^{1}_{u}(\mathcal{B}_{1})$ and $b\in\mathcal{B}_{1}$.
\end{proof}

By Proposition \ref{nablassquare}, we can consider the composition 
\[
(1\otimes_{\nabla_{T}}\delta) \circ\nabla^{S}:\mathcal{X}
\to X^{S}\otimes_{B}^{h}\Omega^{1}_{T}(\mathcal{B}_{1})\otimes^{h}_{\mathcal{B}_{1}}\Omega^{1}_{u}(\mathcal{B}_{1}).
\]
The map $(1\otimes_{\nabla_{T}}\delta) \circ\nabla^{S}$ is 
right $\mathcal{B}_{1}$-linear. This follows from the computation
\begin{align*}
(1\otimes_{\nabla_{T}}\delta) \circ\nabla^{S}(xb)&=(1\otimes_{\nabla_{T}}\delta) (\nabla^{S}(x) b) +(1\otimes_{\nabla_{T}}\delta) (\gamma(x)\otimes \delta(b))\\
&=(1\otimes_{\nabla_{T}}\delta) (\nabla^{S}(x) b) +\nabla_{T} (\gamma(x))\otimes \delta(b)\\
&=(1\otimes_{\nabla_{T}}\delta) (\nabla^{S}(x) )b-(\gamma\otimes\pi_{T}\otimes 1)(\nabla^{S}(x) \otimes\delta b)+\nabla_{T} (\gamma(x))\otimes \delta(b)\\
&=(1\otimes_{\nabla_{T}}\delta) (\nabla^{S}(x) )b,
\end{align*}
where the last line holds since 
$\nabla_{T}(x)=\nabla^{S}_{T}(x)$ in $X^{S}\otimes_{B}^{h}\Omega^{1}_{T}(B,\mathcal{B}_{1})$.
\begin{defn}
\label{def: repcurv}
Let $(\mathcal{X},S)$ be a $C^{1}$-module over a $C^{1}$-Kasparov module $(\mathcal{B},Y,T)$ and $(\nabla,\nabla^{S})$ a $C^{2}$-connection. We define the \emph{represented curvature} of $(\nabla,\nabla^{S})$ to be the operator
$\pi_{T}(\nabla\circ\nabla^{S}):\mathcal{X}\otimes^{\alg}_{\mathcal{B}_{1}}Y\to X^{S}\otimes_{B}Y$ defined on $x\ox y\in \mathcal{X}\otimes^{\alg}_{\mathcal{B}_{1}}Y$ by
\begin{equation}
\label{eq: repcurv}
\pi_{T}(\nabla\circ\nabla^{S})(x\otimes y):= (1\otimes m)(1\otimes 1\otimes\pi_{T})((1\otimes_{\nabla_{T}}\delta) \circ\nabla^{S})(x)y.
\end{equation}
\end{defn}
The notation $\pi_{T}(\nabla\circ\nabla^{S})$ is a convenient shorthand for $(1\otimes m)(1\otimes 1\otimes\pi_{T})((1\otimes_{\nabla_{T}}\delta)\circ\nabla^{S})$.

\subsubsection{The curvature operator of a $C^{2}$-connection on a $C^{(1,2)}$-module}
\label{sect:curv-nabla2}
Additional smoothness simplifies the domain considerations of the previous section. In order to denote the differentiability properties of modules $(\mathcal{X},S)$ in the horizontal and vertical directions we use pairs $(n,k)$ where $k$ corresponds to the horizontal and $n$ to the vertical direction.
\begin{defn}
\label{horverC2mod}
Let $(\mathcal{B},Y,T)$ be a $C^{2}$-Kasparov module. 
A $C^{(1,2)}$-\emph{module} for $(\mathcal{B},Y,T)$ 
is a pair $(\mathcal{X},S)$ such that
such that
\begin{enumerate}
\item $\mathcal{X}$ is a horizontally differentiable $C^{2}$-module with $C^{*}$-closure $X$;
\item $\mathcal{X}\subset\Dom S$.
\end{enumerate}
A $C^{2}$-\emph{connection} on a $C^{(1,2)}$-module $(\mathcal{X},S)$ is a pair $(\nabla,\nabla^{S})$ of connections
\[
\nabla:\mathcal{X}\to X\otimes^{h}_{B}\Omega^{1}_{u}(B,\mathcal{B}_{2}),
\quad \nabla^{S}:\mathcal{X}\to \mathcal{X}_{\nabla_{T}}^{S}
\otimes^{h}_{\mathcal{B}_{1}}
\Omega^{1}_{u}(\mathcal{B}_{1},\mathcal{B}_{2}),
\]
such that $(S\otimes 1,1\otimes_{\nabla}T)$ is a vertically anticommuting pair.
\end{defn}
First observe that if $(\mathcal{X},S)$ is a $C^{(1,2)}$-module then $(\mathcal{X}\otimes_{\mathcal{B}_{2}}^{\alg}\mathcal{B}_{1}, S)$ 
is a $C^{1}$-module for $(\mathcal{B},Y,T)$. For a $C^{2}$-connection on a $C^{(1,2)}$-module $(\mathcal{X},S)$, we wish 
to compute $\pi_{T}(\nabla\circ\nabla^{S})$, for $y\in\Dom T^{2}$.
By Lemmas \ref{indc2} and \ref{lem: XSinclusion} we may define
\begin{align*}
1\otimes_{\nabla^{S}}T^{2}:\mathcal{X}\otimes_{\mathcal{B}_{2}}^{\alg}\Dom T^{2}\to \mathcal{X}^{S}_{\nabla_{T}}\otimes^{h}_{\mathcal{B}_{1}} Y\subset X^{S}\otimes_{B}Y,\quad 
x\otimes y \mapsto x\otimes T^{2}(y)+\nabla_{T^{2}}^{S}(x)y,
\end{align*}
in the spirit of Lemma \ref{symmetric}. Since $T^{2}$ is an even operator, 
the grading $\gamma$ does not appear in this formula. 

\begin{lemma} 
\label{domainproperties}
Let $(\mathcal{B},Y,T)$ be a $C^{2}$ Kasparov module, $(\mathcal{X},S)$ a $C^{(1,2)}$-module and $(\nabla,\nabla^{S})$ a $C^{2}$-connection. Then the operator $$1\otimes_{\nabla}T:\mathcal{X}\otimes_{\mathcal{B}_{1}}^{\alg}\Dom T\to X\otimes_{B}Y,$$ maps $\mathcal{X}\otimes^{\alg}_{\mathcal{B}_{2}}\Dom T^{2}$ into $\Dom (S+i)^{-1}(1\otimes_{\nabla}T)$.
\end{lemma}
\begin{proof} 

For $x\in\mathcal{X}$ and $y\in\Dom T^{2}$ we have
\[
(1\otimes_{\nabla}T)(x\otimes y)=\gamma(x)\otimes Ty+\nabla_{T}(x)y,
\]
and clearly $\gamma(x)\otimes Ty\in \mathcal{X}\otimes^{\alg}_{\mathcal{B}_{1}}\Dom T\subset \Dom (1\otimes_{\nabla}T)\subset \Dom (S+i)^{-1}(1\otimes_{\nabla}T)$. 
Consider 
$$
(S+i)^{-1}\nabla_{T}(x)y=(S+i)^{-1}\cdot\nabla^{S}_{T}(x)y,
$$ 
and observe that 
$\nabla^{S}_{T}(x)y\in \mathcal{X}^{S}_{\nabla_{T}}
\otimes^{h}_{\mathcal{B}_{1}}\Dom T.$
By part 3 of Lemma \ref{domaineq}, it follows that 
\[
(S+i)^{-1}\nabla^{S}_{T}(x)y\in  \Dom 1\otimes_{\nabla}T,
\]
from which we conclude that 
$(S+i)^{-1}\nabla_{T}(x)y\in\Dom 1\otimes_{\nabla}T$. Thus 
$$
\nabla_{T}(x)y\in \Dom (1\otimes_{\nabla}T)(S+i)^{-1}
=\Dom  (S+i)^{-1}(1\otimes_{\nabla}T),
$$
by Lemma \ref{domaineq}.
\end{proof}

It follows by Lemma 
\ref{domaineq}.2 that the operator
\[1\otimes_{\nabla^{S}}T:\mathcal{X}\otimes_{\mathcal{B}_{2}}^{\alg}\Dom T^{2}\to \mathcal{X}^{S}_{\nabla_{T}}\otimes^{h}_{\mathcal{B}_{1}}\Dom T\subset \Dom (1\otimes_{\nabla}T)\subset X^{S}\otimes_{B}Y,\]
is well defined and maps $\mathcal{X}\otimes^{\alg}_{\mathcal{B}_{2}}\Dom T^{2}$ into $\Dom (1\otimes_{\nabla}T)\subset X^{S}\otimes_{B}Y$. 
\begin{defn}
\label{def: curv-operator}
Let $(\mathcal{B},Y,T)$ be a $C^{2}$-Kasparov module, $(\mathcal{X},S)$ a $C^{(1,2)}$-module and $(\nabla,\nabla^{S})$ a $C^{2}$-connection. The \emph{curvature operator} of $(\nabla,\nabla^{S})$ is defined to be the map
\[R_{\nabla_{T}}:\mathcal{X}\otimes^{\alg}_{\mathcal{B}_{2}}\Dom T^{2}\to X^{S}\otimes_{B}Y,\quad R_{\nabla_{T}}:=(1\otimes_{\nabla}T)\circ (1\otimes_{\nabla^{S}}T)-1\otimes_{\nabla^{S}}T^{2}.\]
\end{defn}
By Lemma \ref{domainproperties}, composition with the resolvent $(S+i)^{-1}:X^{S}\to X$ yields the map 
\begin{align*}
(S+i)^{-1}R_{\nabla_{T}}:\mathcal{X}\otimes^{\alg}_{\mathcal{B}_{2}}\Dom T^{2}\to X\otimes_{B}Y,
\end{align*}
which admits the expression $(S+i)^{-1}R_{\nabla_{T}}=(S+i)^{-1}((1\otimes_{\nabla}T)^{2}-1\otimes_{\nabla^{S}}T^{2})$.

Our goal is to identify the curvature operator $R_{\nabla_{T}}$ from Definition \ref{def: curv-operator} with the represented curvature $\pi_{T}(\nabla_{T}\circ\nabla^{S})$ of Equation \eqref{eq: repcurv} in Definition \ref{def: repcurv}.
\begin{lemma} 
\label{computesquare}
Let $(\nabla,\nabla^{S})$ be a $C^{2}$-connection on a $C^{(1,2)}$-module $(\mathcal{X},S)$.  For $\eta\in \mathcal{X}_{\nabla_{T}}^{S}\otimes^{h}_{\mathcal{B}_{1}}\Omega^{1}_{u}(\mathcal{B}_{1},\mathcal{B}_{2})$ and $y\in\Dom T$ it holds that
\begin{align*}
(1\otimes m)(1\otimes 1\otimes\pi_{T})((1\otimes_{\nabla_{T}}\delta)(\eta))y&=(1\otimes_{\nabla}T)(1\otimes \pi_{T})(\eta)y\\
 &\quad\quad\quad\quad +(\gamma\otimes \pi_{T})(\eta)Ty-(\gamma\otimes \pi_{T^{2}})(\eta)y,
\end{align*}
as elements of $X^{S}\otimes_{B}Y$.
\end{lemma}
\begin{proof} Let $\eta:=x\otimes \omega \in \mathcal{X}_{\nabla_{T}}^{S}\otimes^{h}_{\mathcal{B}_{1}}\Omega^{1}_{u}(\mathcal{B}_{1},\mathcal{B}_{2})$ be an elementary tensor.
By Proposition \ref{ajunkie} 
\[m\circ (\pi_{T}\otimes\pi_{T})(\delta\omega)=\delta_{T}\circ\pi_{T}(\omega)-\pi_{T^{2}}(\omega)=[T,\omega_{T}]-\omega_{T^{2}},\] 
and we compute for $y\in \Dom T^{2}$:
\begin{align}
\nonumber (1\otimes m)&(1\otimes 1\otimes\pi_{T})((1\otimes_{\nabla_{T}}\delta)(x\otimes\omega))y 
= (1\otimes m)(\nabla_{T}(x)\omega_{T}
+\gamma(x)\otimes(\pi_{T}\otimes \pi_{T})(\delta\omega))y\\
\nonumber &=\nabla_{T}(x)\omega_{T}y
+\gamma(x)\otimes [T,\omega_{T}]y-\gamma(x)\otimes \omega_{T^{2}}y\\
\nonumber &=\nabla_{T}(x)\omega_{T}y+\gamma(x)\otimes T\omega_{T}y
+\gamma(x)\otimes \omega_{T}Ty-\gamma(x)\otimes \omega_{T^{2}}y\\
\nonumber &=(1\otimes_{\nabla}T)(x\otimes\omega_{T}y)
+\gamma(x)\otimes \omega_{T}Ty-\gamma(x)\otimes \omega_{T^{2}}y\\
\label{nablasquarerep} &=(1\otimes_{\nabla}T)(1\otimes \pi_{T})(x\otimes\omega)y
+(\gamma\otimes\pi_{T})(x\otimes\omega)Ty
-(\gamma\otimes \pi_{T^{2}})(x\otimes\omega)y.
\end{align}
By Lemma \ref{indc2}, $\pi_{T}$ defines a continuous map
\[
1\otimes \pi_{T}:\mathcal{X}^{S}_{\nabla_{T}}\otimes^{h}_{\mathcal{B}_{1}}\Omega^{1}_{u}(\mathcal{B}_{1},\mathcal{B}_{2})\to \mathbb{B}(\Dom T,\mathcal{X}^{S}_{\nabla_{T}}\otimes_{\mathcal{B}_{1}}^{h}\Dom T)\subset \mathbb{B}(\Dom T ,\Dom 1\otimes_{\nabla}T),
\]
and a continuous map
\[
\gamma\otimes \pi_{T}:\mathcal{X}^{S}_{\nabla_{T}}\otimes^{h}_{\mathcal{B}_{1}}\Omega^{1}_{u}(\mathcal{B}_{1},\mathcal{B}_{2})\to \mathbb{B}(\Dom T, X\otimes_{B}Y).
\]
Invoking Lemma \ref{lem: XSinclusion} as well, we see that $\pi_{T^{2}}$ defines a continuous map
\[
\gamma\otimes \pi_{T^{2}}:\mathcal{X}^{S}_{\nabla_{T}}\otimes^{h}_{\mathcal{B}_{1}}\Omega^{1}_{u}(\mathcal{B}_{1},\mathcal{B}_{2})\to\mathbb{B}(\Dom T,\mathcal{X}^{S}_{\nabla_{T}}\otimes_{\mathcal{B}_{1}}^{h} Y)\to \mathbb{B}(\Dom T,X^{S}\otimes_{B} Y).
\]
Therefore by Equation \eqref{nablasquarerep} the operator $ (1\otimes 1\otimes\pi_{T})(1\otimes_{\nabla_{T}}\delta)$
extends by continuity to all $\eta\in \mathcal{X}^{S}_{\nabla_{T}}\otimes_{\mathcal{B}_{1}}^{h}\Omega^{1}_{u}(\mathcal{B}_{1},\mathcal{B}_{2})$ and, as $\Dom T^{2}$ is a core for $T$, to all $y\in\Dom T$.
\end{proof}
For a $C^{2}$-connection $(\nabla,\nabla^{S})$ on a $C^{(1,2)}$-module $(\mathcal{X},S)$ and $x\in\mathcal{X}$ we have $\nabla_{T}^{S}(x)\in\mathcal{X}^{S}_{\nabla_{T}}\otimes^{h}_{\mathcal{B}_{1}}\Omega^{1}_{u}(\mathcal{B}_{1},\mathcal{B}_{2})$, which maps completely contractive to $\mathcal{X}^{S}_{\nabla_{T}}\otimes^{h}_{\mathcal{B}_{1}}\Omega^{1}_{u}(\mathcal{B}_{1})$. Hence by Lemma \ref{computesquare} the operator $\pi_{T}(\nabla\circ\nabla^{S}):\mathcal{X}\otimes^{\alg}_{\mathcal{B}_{1}}Y\to X^{S}\otimes_{B}Y$ is defined.

We then come to our main result.
\begin{thm}
\label{thm: Ristwoformdetermineduptojunk}
Let $(\mathcal{B},Y,T)$ be a $C^{2}$-Kasparov module, $(\mathcal{X},S)$ a $C^{(1,2)}$-module and $(\nabla,\nabla^{S})$ a $C ^{2}$-connection.
Let $\pi_{T}(\nabla\circ\nabla^{S})$ be the represented curvature from Definition \ref{def: repcurv} and $R_{\nabla_T}$ be the curvature operator from Definition \ref{def: curv-operator}.
Then there is an equality 
\[
R_{\nabla_{T}}
=\pi_{T}(\nabla\circ\nabla^{S}):\mathcal{X}\otimes^{\alg}_{\mathcal{B}_{2}}\Dom T^{2}\to X^{S}\otimes_{B}Y.
\]
Consequently, $R_{\nabla_{T}}(x)\in X^{S}\otimes_{B}^{h}\Omega^{2}_{T}(\B_{1})\subset\mathbb{B}(Y,X^{S}\otimes_{B}Y)$ for all $x\in\mathcal{X}$ and $R_{\nabla_{T}}$ extends to $\mathcal{X}\otimes^{\alg}_{\mathcal{B}_{2}}Y$. If $(\nabla',\nabla'^{S})$ is another $C^{2}$-connection with 
$\nabla_{T}=\nabla'_{T}$, then the difference
$(R_{\nabla_{T}}-R_{\nabla'_{T}})(x)\in X^{S}\otimes^{h}_{B}\mathcal{J}^{2}_{T}(\B_{1})$.
\end{thm}
\begin{proof}
For $x\in\mathcal{X}$, $\nabla^{S}(x)\in\mathcal{X}^{S}_{\nabla_{T}}\otimes^{h}_{\mathcal{B}_{1}}\Omega^{1}_{u}(\mathcal{B}_{1},\mathcal{B}_{2})$, so we compute using Lemma \ref{computesquare}, with $y\in\Dom T^{2}$:
\begin{align}
\nonumber\pi_{T}(\nabla\circ\nabla^{S})(x\otimes y)&=(1\otimes m)(1\otimes 1\otimes\pi_{T})((1\otimes_{\nabla_{T}}\delta)\circ \nabla^{S})(x)y\\
\label{reprep}&=(1\otimes_{\nabla}T)\nabla_{T}^{S}(x)y+\nabla_{T}^{S}(x)Ty-\nabla_{T^{2}}^{S}(x)y\\
\nonumber&=(1\otimes_{\nabla}T)\nabla_{T}^{S}(x)y+\nabla_{T}^{S}(x)Ty+x\otimes T^{2}y -\nabla_{T^{2}}^{S}(x)y-x\otimes T^{2}y\\
\nonumber&=(1\otimes_{\nabla}T)(\nabla_{T}^{S}(x)y + x\otimes Ty)-(1\otimes_{\nabla^{S}}T^{2})(x\otimes y)\\
\nonumber&=( (1\otimes_{\nabla}T)\circ (1\otimes_{\nabla^{S}}T)-1\otimes_{\nabla^{S}}T^{2})(x\otimes y)\\
\nonumber&=R_{\nabla_{T}}(x) y
\end{align}

as desired. Now if $(\nabla',\nabla'^{S})$ is such that $\nabla_{T}=\nabla'_{T}$ then Equation \eqref{reprep} gives
\begin{align*}
(R_{\nabla_{T}}-R_{\nabla'_{T}})(x\otimes y)=(\nabla_{T^{2}}^{S}-\nabla'^{S}_{T^{2}})(x)y=(1\otimes\pi_{T^{2}})(\nabla^{S}-\nabla'^{S})(x)y.
\end{align*}
By Corollary \ref{cor:junk-squared} $(1\otimes\pi_{T^{2}})(\nabla^{S}-\nabla'^{S})(x)\in\mathcal{X}^{S}_{\nabla_{T}}\otimes^{h}_{\mathcal{B}_{1}}\mathcal{J}^{2}_{T}(\mathcal{B}_{1})=X^{S}\otimes^{h}_{B}\mathcal{J}^{2}_{T}(\mathcal{B}_{1})$.
\end{proof}

\subsection{The curvature of a $C^{2}$-correspondence}
\label{sect:curv-corresp}

So far we have focused on giving meaning to the curvature of a $C^{2}$-connection on $X$ under minimal differentiability assumptions. 
The represented curvature exists on $C^{1}$-modules, and it coincides with the curvature operator on $C^{(1,2)}$-modules.
It is now time to give meaning to Equation \eqref{correspondencecurve-intro} and define the curvature associated to a correspondence. 
Briefly, a correspondence is a Kasparov module together with a 
connection on the module, and we define this in detail below. Before
doing so, we describe how we impose additional $C^2$-conditions on our modules and connections.

\begin{defn}
\label{horverC2}
Let $(\mathcal{B},Y,T)$ be a $C^{2}$-Kasparov module. 
A $C^{2}$-\emph{module} for $(\mathcal{B},Y,T)$ 
is a pair $(\mathcal{X},S)$ 
such that
\begin{enumerate}
\item $\mathcal{X}$ is a horizontally differentiable $C^{2}$-module with $C^{*}$-closure $X$;
\item $\mathcal{X}\subset\Dom S^{2}$.
\end{enumerate}
A $C^{2}$-\emph{connection} on a $C^{2}$-module $(\mathcal{X},S)$ is a pair $(\nabla,\nabla^{1})$ of connections
\[
\nabla:\mathcal{X}\to X\otimes^{h}_{B}\Omega^{1}_{u}(B,\mathcal{B}_{2}),
\quad \nabla^{1}:\mathcal{X}\to \mathcal{X}_{\nabla_{T}}
\otimes^{h}_{\mathcal{B}_{1}}
\Omega^{1}_{u}(\mathcal{B}_{1},\mathcal{B}_{2}),
\]
such that $(S\otimes 1,1\otimes_{\nabla}T)$ is a vertically anticommuting pair.
\end{defn}
Let us clarify the modifications of the above definition relative to Definitions \ref{def: connpair} and \ref{horverC2}. The content of 
conditions 1. and 2. is that the module $\mathcal{X}$ is now assumed to be both horizontally and vertically $C^{2}$ as opposed to just vertically $C^1$. The curvature of the
$C^{2}$-connection is now viewed, using our Sobolev space analogy from Remark \ref{rmk:sobo-rmk}, as a map $W^{2}_{(2,2)}\to L^{2}$, whereas with only a $C^{1}$-structure it is a map $W^{2}_{(1,2)}\to W_{(-1,0)}^{2}$. The next Lemma makes this statement precise.
\begin{lemma}  Let $(\mathcal{X},S)$ be a $C^{2}$-module and $(\nabla,\nabla^{1})$ a $C^{2}$-connection. The connection \[\nabla^{S}:=(\iota_{S}\otimes 1)\circ\nabla^{1}:\mathcal{X}\to\mathcal{X}_{\nabla_{T}}^{S}\otimes^{h}_{\mathcal{B}_{1}}\Omega^{1}_{u}(\mathcal{B}_{1},\mathcal{B}_{2}),\] makes $(\nabla,\nabla^{S})$ into a $C^{2}$-connection on $(\mathcal{X},S)$ viewed as a $C^{(1,2)}$-module.
\end{lemma}
\begin{proof} The identity map on $\mathcal{X}$ extends to a complete contraction $\iota_{S}:\mathcal{X}_{\nabla_{T}}\to \mathcal{X}_{\nabla_{T}}^{S}$.
The only thing to check from Definition \ref{def: connpair} is the compatibility $(S+i)^{-1}\nabla_{T}(x)=(S+i)^{-1}\cdot\nabla^{S}_{T}(x)$, which holds automatically.
\end{proof}

The next proposition shows that the curvature operator of
a $C^2$-connection is well-defined.
\begin{prop}
\label{prop: C2curvedomain}
Let $(\mathcal{X},S)$ be a $C^{2}$-module and $(\nabla,\nabla^{1})$ a $C^{2}$-connection. Then 
\[\mathcal{X}\otimes^{\alg}_{\mathcal{B}_{2}}\Dom T^{2}\subset \Dom (S\otimes 1+ 1\otimes_{\nabla}T)^{2}\subset \mathcal{F}(S\otimes 1,1\otimes_{\nabla}T),\]
and $R_{\nabla_{T}}:=(1\otimes_{\nabla}T)^{2}-1\otimes_{\nabla}T^{2}$ is a symmetric operator defined on $\mathcal{X}\otimes_{\mathcal{B}_{2}}^{\alg}\Dom T^{2}$ that extends to $\mathcal{X}\otimes^{\alg}_{\mathcal{B}_{2}}Y$. Here the set $\mathcal{F}$ is as in Definition \ref{wac}.
\end{prop}
\begin{proof}
As $(S\otimes 1,1\otimes_{\nabla}T)$ is a vertically anticommuting pair \cite[Theorem 5.1]{LM} gives that $\Dom (S^{2}\otimes 1)\cap \Dom (1\otimes_{\nabla}T)^{2}=\Dom (S\otimes 1+1\otimes_{\nabla}T)^{2}\subset \mathcal{F}(S\otimes 1,1\otimes_{\nabla}T)$. By Lemma \ref{domainproperties} $$\mathcal{X}\otimes^{\alg}_{\mathcal{B}_{2}}\Dom T^{2}\subset \Dom (1\otimes_{\nabla}T)^{2},$$ so  condition 2 of Definition \ref{horverC2} implies that $\mathcal{X}\otimes_{\mathcal{B}_{2}}^{\alg}\Dom T^{2}\subset \Dom (S\otimes 1+1\otimes_{\nabla}T)^{2}$.
\end{proof}
\begin{defn}
A $C^{2}$-\emph{correspondence} between a 
$C^2$-Kasparov $A$-$C$ module $(\A,E,D)$ and a 
$C^2$-Kasparov  $B$-$C$ module
$(\B,Y,T)$ 
is a quintuple $(\A,X,S,(\nabla,\nabla^{1}))$ such that:
\begin{enumerate}
\item $(\A,X,S)$ is an unbounded $(A,B)$ Kasparov module;
\item $(\mathcal{X},S)$ is a $C^{2}$-module;
\item $(\nabla,\nabla^{1})$ is a $C^{2}$-connection;
\item there is a unitary isomorphism $u:E\to X\otimes_{B}Y$ intertwining the $\A$-representations and such that 
$$
\Dom D\cap u^{*}(\Dom S\otimes 1\cap\Dom 1\otimes_{\nabla}T),
$$ 
is dense in $E$ and 
\[
Z:=D-u^{*}(S\otimes 1+1\otimes_{\nabla}T)u:\Dom D\cap 
u^{*}(\Dom S\otimes 1\cap\Dom 1\otimes_{\nabla}T)\to E,
\]
is bounded and preserves $\Dom D$.
\end{enumerate}
\end{defn}
As $Z$ preserves $\Dom D$, $(D,Z)$ is a vertically anticommuting pair. By \cite[Theorem 5.1]{LM}
\[
\Dom D^{2}=\Dom D^{2}\cap\Dom Z^{2}=\Dom (D-Z)^{2}=\Dom u^{*}(S\otimes 1+1\otimes_{\nabla}T)^{2}u.
\]
Then by Definiton \ref{horverC2}, for a $C^2$-correspondence we have $\mathcal{X}\otimes^{\alg}_{\mathcal{B}_{2}}\Dom T^{2}\subset \Dom D^{2}$.

As mentioned, correspondences are more-or-less unbounded Kasparov modules 
with additional connection data. The additional connection data
allows one to construct representatives of the Kasparov product \cite{BMS,KaLe2,LM,Mes,MR}, as
described in the next result.
\begin{prop}
Let $(\A,\mathcal{X},S,(\nabla,\nabla^{1}))$ be a $C^{2}$-correspondence 
for $(\A,E,D)$ and $(\B,Y,T)$. Then
\[
[(\A,E,D)]=[(\A,X,S)]\otimes_B [(\B, Y, T)]\in KK_{0}(A,C),
\]
where $\otimes_B$ denotes the Kasparov product.
\end{prop}
\begin{proof}
Since $Z$ is  bounded it follows from 
\cite[Theorem 4.2]{vdDungen} that $D$ and $D-Z$ 
represent the same $KK$-class. Since 
$(S\otimes 1,1\otimes_{\nabla}T)$ form a weakly 
anti-commuting pair and $1\otimes_{\nabla}T$ commutes 
boundedly with $\A$, it follows from \cite[Theorem 7.2]{LM} 
that the product relation holds.
\end{proof}
The reason for introducing correspondences is the more refined curvature data that becomes available.
\begin{defn}
Let $(\A,E,D)$ and $(\B,Y,T)$ be two $C^{2}$-Kasparov modules and 
let $(\A,\mathcal{X},S,(\nabla,\nabla^{1}))$ a $C^{2}$-correspondence for them. We define the \emph{curvature operator} of 
the correspondence
$(\A,\mathcal{X},S,(\nabla,\nabla^{1}))$ to be the symmetric operator
\begin{equation}
\label{correspondencecurve}
R_{(S,\nabla_{T})}:=(S\otimes 1+1\otimes_{\nabla}T)^{2}-S^{2}\otimes 1
-1\otimes_{\nabla}T^{2}:\mathcal{X}\otimes_{\mathcal{B}_{2}}^{\alg}\Dom T^{2}\to X\otimes_{B} Y.
\end{equation}
\end{defn}
By Proposition \ref{prop: C2curvedomain} it holds that $R_{(S,\nabla_{T})}=R_{\nabla_{T}}+[S\otimes 1,1\otimes_{\nabla}T]_+$, 
by a straightforward algebraic calculation on the domain $\mathcal{X}\otimes_{\mathcal{B}_{2}}^{\alg}\Dom T^{2}$.

Note that it is not necessary to specify $D$ in order to define the curvature operator, since we could just as well take $D$ to be the tensor sum $S \otimes 1 + 1 \otimes_{\nabla} T$ itself. However, in examples it turns out that the bounded operator $Z$ appearing as their difference contains geometric information as well (see for instance Equation \eqref{eq:factorization}).

 \subsection{Universal connections on locally convex modules}
 \label{sec:uni-conn}

In the applications of our theory ({\em cf.} Section \ref{sect:riem-subm} below) one typically finds that the modules actually come equipped with more differentiable structure, beyond  the $C^2$-stucture described above. For instance, in the category of smooth manifolds the Hilbert modules are typically based on Fr\'echet modules, and Fr\'echet continuous maps. Let us describe here how to incorporate such locally convex spaces and algebras in the above $C^2$-context.

Let $(\mathcal{B},Y,T)$ be a $C^{2}$-Kasparov module and assume that $\mathcal{B}$ carries the structure of a complete locally convex $m$-$*$-algebra for which there is a continuous inclusion
$\mathcal{B}\to \mathcal{B}_{2}$, see \cite{Frag}. Denote by $B$ the 
$C^{*}$-closure of $\mathcal{B}$ in the norm coming 
from $\mathbb{B}(Y)$. The Haagerup tensor norm is a \emph{cross-norm}, that is, it satisfies $\|x\otimes y\|_{h}= \|x\|\|y\|$
(see \cite[Section 1.5.4]{BleLeM}). Denoting the projective tensor product by $\widehat{\ox}$, \cite[Propositions 43.4 and 43.12.a)]{Treves} 
prove that the identity map on $\mathcal{B}\otimes^{\alg}\mathcal{B}$ 
extends to continuous maps 
\[
\mathcal{B}\widehat{\otimes}\mathcal{B}\to B\widehat{\otimes}\mathcal{B}_{k}\to B\otimes^{h}\mathcal{B}_{k},\quad \mathcal{B}\widehat{\otimes}\mathcal{B}\to \mathcal{B}_{1}\widehat{\otimes}\mathcal{B}_{k}\to \mathcal{B}_{1}\otimes^{h}\mathcal{B}_{k},\quad k=1,\,2.
\]
By continuity of the multiplication maps, these inclusions 
restrict to continuous maps
\begin{equation}
\label{eq: forminclusion}
\Omega^{1}_{u}(\mathcal{B})\to \Omega^{1}_{u}(B,\mathcal{B}_{k}),\quad \Omega^{1}_{u}(\mathcal{B})\to \Omega^{1}_{u}(\mathcal{B}_{1},\mathcal{B}_{2}).
\end{equation}
Let $\mathcal{X}$ be a locally convex topological vector space which is a right $\mathcal{B}$-module such that the module multiplication defines a continuous map $\mathcal{X}\widehat{\otimes}\mathcal{B}\to \mathcal{X}$. Moreover assume that there is a continuous inner product 
\[
\mathcal{X}\times\mathcal{X}\to \mathcal{B},
\]
giving $\mathcal{X}$ the structure of a 
pre-Hilbert $C^{*}$-module over the 
pre-$C^{*}$-algebra $\B$ and denote by $X$ 
the $C^{*}$-module closure of $\X$. The identity on $\mathcal{X}$ 
induces a continuous map $\mathcal{X}\to X$, 
and thus by \cite[Proposition 43.4]{Treves} we obtain  continuous maps
\begin{align}
\iota_{0}:
\mathcal{X}\widehat{\otimes}_{\mathcal{B}}\Omega^{1}_{u}(\mathcal{B})&\to X\widehat{\otimes}_{B}\Omega^{1}_{u}(B,\mathcal{B}_{1})
\to X\otimes_{B}^h\Omega^{1}_{u}(B,\mathcal{B}_{1})\nonumber\\
\iota_{1}:
\mathcal{X}\widehat{\otimes}_{\mathcal{B}}\Omega^{1}_{u}(\mathcal{B})&\to \X\widehat{\otimes}_{\B_1}\Omega^{1}_{u}(\B_1,\mathcal{B}_{2})
\to \X_{\nabla_T}\otimes_{\B_1}^h\Omega^{1}_{u}(\B_1,\mathcal{B}_{2}).
\label{eq:the-map-known-as-iota}
\end{align}
The maps $\iota_{0},\,\iota_1$ are well-defined because the Haagerup norm is a cross-norm and the projective norm is the largest cross-norm \cite[Proposition 43.12.a)]{Treves}. Therefore
\begin{equation}
\iota_{T}:=(1\otimes\pi_{T})\circ \iota_0:
\mathcal{X}\widehat{\otimes}_{\mathcal{B}}\Omega^{1}_{u}(\mathcal{B})
\to X\otimes_{B}^h\Omega^{1}_{T}(B,\mathcal{B}_{1}),
\label{eq:iota-tee}
\end{equation} 
is continuous as well. 

\begin{prop}
\label{locconvSconn}
Let $(\mathcal{B},Y,T)$ be a $C^{2}$-Kasparov module.
Assume that $\nabla_{u}:\mathcal{X}\to \mathcal{X}\widehat{\otimes}_{\mathcal{B}}\Omega^{1}_{u}(\mathcal{B})$ is a Hermitian connection, $S:\mathcal{X}\to X$ an essentially self-adjoint and regular operator and $(S\otimes 1, 1\otimes_{\nabla}T)$ is a vertically anticommuting pair. If the map
\[(S+i)^{-1}\circ\iota_{T}\circ \nabla_{u}:\mathcal{X}\to X\otimes^{h}_{B}\Omega^{1}_{T}(B,\mathcal{B}_{1}),\] is continuous, then the identity map $\mathcal{X}\to \mathcal{X}^{S}_{\nabla_{T}}$ is continuous.
Consequently the identity map on the algebraic tensor product extends to a continous map
\[
\iota_{S}:\mathcal{X}\widehat{\otimes}_{\mathcal{B}}\Omega^{1}_{u}(\mathcal{B})\to \mathcal{X}^{S}_{\nabla_{T}}\otimes^{h}_{\mathcal{B}_{1}}\Omega^{1}_{u}(\mathcal{B}_{1},\mathcal{B}_{2}).\]
With $\nabla:=\iota_{0}\circ\nabla_{u}$ and $\nabla^{S}:=\iota_{S}\circ\nabla_{u}$, the pair $(\nabla,\nabla^{S})$ is a $C^{2}$-connection on the $C^{(1,2)}$-module $(\mathcal{X}\otimes_{\mathcal{B}}^{\alg}\mathcal{B}_{2},S)$.
\end{prop}
\begin{rmk}
When $S=0$ the map $\iota_S$ reduces to the map $\iota_1$ defined above.
\end{rmk}
\begin{proof}
Since we have assumed that $(S\otimes 1, 1\otimes_{\nabla}T)$ is a vertically anti-commuting pair, it suffices to define $\nabla^{S}$ and verify that $\nabla^S$ satisfies condition 2. of Definition \ref{def: connpair}.

 Continuity of the map $(S+i)^{-1}\circ \iota_{T}\circ\nabla_{u}$ means that for the Haagerup norm $\|\cdot \|_{X\otimes^{h}_{B}\Omega^{1}_{T}(B,\mathcal{B}_{1})}$ 
 there is a continuous seminorm $p$ on $\mathcal{X}$ such that
\begin{equation}
\label{locconvcont}
\|(S+i)^{-1}\circ\iota_{T}\circ\nabla_{u}(x)\|_{X\otimes^{h}_{B}\Omega^{1}_{T}(B,\mathcal{B}_{1})} \leq p(x). 
\end{equation}
Thus we obtain a continuous inclusion $\iota_{S}:\mathcal{X}\to \mathcal{X}^{S}_{\nabla_{T}}$. The remaining statements now follow by functoriality of the projective tensor product for continuous maps \cite[Proposition 43.4]{Treves}. For the pair $(\nabla,\nabla^{S}):=(\iota_{0}\circ\nabla_{u}, \iota_{S}\circ\nabla_{u})$ and $x\in\mathcal{X}$ it holds that 
\begin{align*}
  (S+i)^{-1}\nabla(x)&=(S+i)^{-1}(\iota_{0}\circ\nabla_{u})(x)\\ &=(S+i)^{-1}\cdot (\iota_{S}\circ\nabla_{u})(x)=(S+i)^{-1}\nabla^{S}(x),
  \end{align*}
  and thus condition 2. of Definition \ref{def: connpair} is satisfied. 
\end{proof}
\begin{corl}
\label{universalsquare}
Assume that 
$\nabla_{u}:\mathcal{X}\to \mathcal{X}\widehat{\otimes}_{\mathcal{B}}\Omega^{1}_{u}(\mathcal{B})$ is a universal Hermitian connection such that the connection $\iota_{T}\circ\nabla_{u}:\mathcal{X}\to X\otimes^{h}_{B}\Omega^{1}_{T}(B,\mathcal{B}_{1})$ is continuous.
Applying Proposition \ref{locconvSconn} with $S=0$ yields 
the $C^{2}$-connection 
\[
\nabla^{1}:=\iota_{1}\circ\nabla_{u}:\mathcal{X}\to \mathcal{X}_{\nabla_T}\otimes^{h}_{\mathcal{B}_{1}}\Omega^{1}_{u}(\mathcal{B}_{1},\mathcal{B}_{2}),
\] 
on the $C^{2}$-module $(\mathcal{X},S)$. Assume further that
\begin{enumerate}
\item $\mathcal{X}\subset\Dom S^{2}$; 
\item $(\A,X,S)$ is an unbounded Kasparov module;
\item $(S\otimes 1, 1\otimes_{\nabla}T)$ is a vertically anti-commuting pair;
\item for all $a\in\mathcal{A}$, $a:\Dom 1\otimes_{\nabla}T\to \Dom 1\otimes_{\nabla}T$ and $[1\otimes_{\nabla}T,a]$ is bounded.
\end{enumerate}

Then
$(\A,\mathcal{X}\otimes^{\alg}_{\mathcal{B}}\mathcal{B}_{2}, S, (\iota_{0}\circ\nabla_{u},\iota_{1}\circ\nabla_{u}))$ is a $C^{2}$-correspondence for $(\A, X\otimes_{B}Y, S\otimes 1+1\otimes_{\nabla}T)$ and $(\mathcal{B},Y,T)$. On $\mathcal{X}\otimes_{\mathcal{B}}^{\alg}\Dom T^{2}$ there are equalities
\begin{align}
\label{projectivecurvature}
R_{\nabla_{T}}&=(1\otimes_{\nabla}T)^{2}-1\otimes_{\nabla}T^{2}=(1\otimes m)(1\otimes\pi_{T}\otimes\pi_{T})(\nabla_{u}^{2}),\\ R_{(S,\nabla_{T})}&=R_{\nabla_{T}}+[S\otimes 1,1\otimes_{\nabla}T]_+,
\end{align}
of symmetric operators in $X\otimes_{B}Y$.
\end{corl}
\begin{proof} 
The equality \eqref{projectivecurvature} is proved analogously to Lemma \ref{computesquare} and Theorem \ref{thm: Ristwoformdetermineduptojunk}.
\end{proof}
\section{Finitely generated projective modules over spectral triples}
\label{sect:fgp}
In this section we assume that $B$ is a unital $C^{*}$-algebra 
and $X$ is a finitely generated full Hilbert $C^{*}$-module over $B$. 
Then $X$ is algebraically finitely generated and projective. 
We assume that $X$ is $\mathbb{Z}/2$-graded with 
grading $\gamma$.
We describe the 
construction of the curvature operator explicitly 
in this case.
Our results are in complete agreement with the purely algebraic approach described, for instance, in \cite{Landi}.

\subsection{Connections on finite projective $C^{2}$-modules}

We fix a $C^{2}$-spectral triple $(\mathcal{B},H,D)$ 
in the sense of Definition \ref{C2spec}. Consider a $\mathbb{Z}_2$-graded inner product module 
$\mathcal{X}$ which is finitely generated
and projective over $\mathcal{B}_{2}$, together with an inner-product preserving injection $v:\mathcal{X}\to \mathcal{B}_{2}^{2N}$ of right $\mathcal{B}_{2}$-modules. The map $v$ induces a module isomorphism $v:\mathcal{X}\to p\mathcal{B}^{2N}_{2}$, where $p\in M_{2N}(\mathcal{B}_{2})$ is a projection and extends to an isometry $v:X\to pB^{2N}$ on the $C^{*}$-module level. 

The main simplification of the construction in Section \ref{Hilbcurve} is that we take the vertical operator $S=0$. Most importantly, for
a Hermitian connection $\nabla:\mathcal{X}\to X\otimes_{B}^{h}\Omega^{1}_{u}(B,\mathcal{B}_{1})$ 
we have 
$\mathcal{X}^{S}_{\nabla_{D}}=\mathcal{X}_{\nabla_{D}}$. Considering $v$ as a map $v:\mathcal{X}\to \mathcal{B}_{1}^{2N}$ via the inclusion $\mathcal{B}_{2}\to\mathcal{B}_{1}$,
we define the norm  $\|x\|_{v} :=\|v(x)\|_{\mathcal{B}^{2N}_{1}}$  on $\mathcal{X}$ and denote the completion of $\mathcal{X}$ in this norm by $\mathcal{X}_{v}$. 

Let $e_i$ denote the standard basis of $\mathcal{B}_{2}^{2N}$ and set $x_{i}:=v^{*}(e_i)$. Then the finite set $\{x_i\}_{1\leq |i|\leq N}\subset\X$ is a frame for $\mathcal{X}$, that is  
${\rm Id}_{\mathcal{X}}=\sum_{1\leq |i|\leq N} |x_i\rangle\langle x_i|$. 

The \emph{Grassmann connection}
\[\nabla^{v}:\mathcal{X}\to \mathcal{X}\otimes^{\alg}_{\mathcal{B}_{2}}\Omega^{1}_{u}(\mathcal{B}_{1},\mathcal{B}_{2}),\quad \nabla^{v}(x):=\sum_{1\leq |i|\leq N}\gamma(x_{i})\otimes \delta(\langle x_{i}, x\rangle),\]
is a well-defined Hermitian connection. Given any other connection
$\nabla$ on $\mathcal{X}$ we define the connection one-form
$\omega(x):=\nabla(x)-\nabla^{v}(x)$.

Recall from Equation \eqref{eq:ex-nabla-tee-ess} that for a Hermitian connection $\nabla$ and $S=0$ we have $\|x\|_{\nabla_{D}} =\max\{\|x\|_{X},\|\nabla_{D}(x)\|_{X\otimes^{h}_{B}\Omega^{1}_{D}}\}$.
\begin{lemma} Let $\nabla:\mathcal{X}\to X\otimes_{B}\Omega^{1}_{u}(B,\mathcal{B}_{1})$ be a Hermitian connection on $\mathcal{X}$. Then the operator space norms $\|\cdot\|_{v}$ and $\|\cdot \|_{\nabla_{D}}$  are cb-equivalent. 
\end{lemma}
\begin{proof}
It follows from \cite[Lemma 3.6]{MR} that the norms $\|x\|_{v} =\|v(x)\|_{\mathcal{B}^{2N}_{1}}$ 
and $\|x\|_{\nabla^{v}}$ are cb-equivalent. Since 
$\omega(x):=\nabla(x)-\nabla^{v}(x)$ is $\mathcal{B}_{2}$-linear it is bounded for the $C^{*}$-norm. It follows that the norm $\|x\|_{\nabla_{D}}=\max\{\|x\|_{X},\|\nabla_{D}(x)\|_{X\otimes^{h}_{B}\Omega^{1}_{D}}\}$ is equivalent to the norm $\|x\|_{v}$. 
\end{proof}
This means that for any two connections $\mathcal{X}_{\mathcal{\nabla}}=\mathcal{X}_{\nabla'}=\mathcal{X}_{v}$. The operator module $\mathcal{X}_{v}$ is finitely generated
and projective over $\mathcal{B}_{1}$, as is $X$ over $B$. For any $w\in \mathcal{X}_{v}$ the map 
\[\langle w|:\mathcal{X}_{v}\to \mathcal{B}_{1},\quad x\mapsto \langle w,x\rangle,\]
is completely bounded by \cite[Proposition 3.7.2]{MR}. Hence  for any left operator $\mathcal{B}_{1}$-module $Z$ and 
$\omega\in\mathcal{X}_{v}\otimes^{h}_{\mathcal{B}_{1}}Z$ we then have $\omega=\sum_{1\leq |i|\leq N} x_{i}\otimes \langle x_{i},\omega\rangle$ so $\mathcal{X}_{v}\otimes^{h}_{\mathcal{B}_{1}}Z=\mathcal{X}_{v}\otimes^{\alg}_{\mathcal{B}_{1}}Z$. Similarly $X\otimes^{h}_{B}Z=X\otimes^{\alg}_{B}Z$.
We summarise the simplifications in the following definition.
\begin{defn} 
Let $(\mathcal{B},H,D)$ a $C^{2}$-spectral triple. For $k=1,2$ a $C^{k}$-\emph{submodule} of $X$ is a dense  $\mathcal{B}_{k}$ submodule $\mathcal{X}\subset X$ such that for all $x_{1},x_{2}\in\mathcal{X}$ we have $\langle x_{1},x_{2}\rangle \in \mathcal{B}_{k}$.
 A \emph{$C^{2}$-connection} is a linear map $\nabla:\mathcal{X}\to \mathcal{X}_{v}\otimes_{\mathcal{B}_{1}}^{\alg}\Omega^{1}_{u}(\mathcal{B}_{1},\mathcal{B}_{2})$ satisfying the Leibniz rule. 
\end{defn}


Given a $C^2$-connection $\nabla$ on the finite projective module
$\X$, we obtain the represented connection
$$
\nabla_D:\,\X \to X\otimes_{B}^{\alg}\Omega^{1}_{D}(B, \mathcal{B}_{1}),\quad \nabla_D(x)=(1\ox \pi_D)(\nabla(x)).
$$

Similarly we obtain
\[\nabla_{D^{2}}:\mathcal{X}\to \mathcal{X}_{v}\otimes_{\mathcal{B}_{1}}^{\alg}\Omega^{1}_{D^{2}}(\mathcal{B}_{2},\mathcal{B}_{1}),\]
defined by $\nabla_{D^{2}}(x):=(1\ox\pi_{D^{2}})(\nabla(x))$.

\subsection{The curvature operator for finitely generated projective modules}


We are now in a position to apply our general formalism to compute the curvature operator of a finite projective module. We first recall the following well-known result.

\begin{prop} 
\label{prop:prod-op}
Suppose that $\mathcal{X}$ is algebraically finitely generated and projective over $\mathcal{B}_{2}$.
The operator
\[
1\otimes_{\nabla}D:\mathcal{X}\otimes^{\alg}_{\mathcal{B}_{2}}\Dom D
\to X\otimes_{B}H,\quad 
x\otimes h\mapsto \gamma(x)\otimes Dh+\nabla_{D}(x)h,
\]
is well-defined and essentially self-adjoint. Moreover 
$\mathcal{X}\otimes_{\mathcal{B}_{2}}^{\alg}\Dom D^{2}
\subset \Dom (1\otimes_{\nabla}D)^{2}$ and the operator 
\[
1\otimes_{\nabla}D^{2}:\mathcal{X}\otimes^{\alg}_{\mathcal{B}_{2}}
\Dom D^{2}\to X\otimes_{B}H,\quad 
x\otimes h\mapsto x\otimes D^{2}h+\nabla_{D^{2}}(x)h,
\]
is well-defined and symmetric. The curvature operator\[
R_{\nabla_{D}}:=1\otimes_{\nabla}D^{2}
-(1\otimes_{\nabla}D)^{2}:\mathcal{X}\otimes^{\alg}_{\mathcal{B}_{2}}\Dom D^{2}
\to X\otimes_{B}H,
\]
is a densely defined symmetric operator.
\end{prop}
\begin{proof} 
The first statement is proved in \cite{BRB,ConnesGrav,LRV} and several subsequent 
works \cite{BMS,KaLe2, Mes, MR}. The second statement follows from Lemma \ref{domainproperties} and the third statement from Theorem \ref{thm: Ristwoformdetermineduptojunk}  both with $S=0$.
\end{proof}



\begin{prop}
Suppose that $\mathcal{X}$ is finitely generated and projective over $\mathcal{B}_{2}$ and that $\nabla:\mathcal{X}\to \mathcal{X}_{v}\otimes^{\alg}_{\mathcal{B}_{1}}\Omega^{1}_{u}(\mathcal{B}_{1},\mathcal{B}_{2})$ is a $C^{2}$-connection. Then $R_{\nabla}$ extends to a bounded operator on $X\otimes_{B}H$. Moreover if $\nabla^{v}$ is the Grassmann connection of the frame $\{x_{i}\}$ and 
\[\omega:\mathcal{X}\to \mathcal{X}\otimes^{\alg}_{\mathcal{B}_{1}}\Omega^{1}_{u}(\mathcal{B}_{1},\mathcal{B}_{2}),\quad \omega(x):=\nabla(x)-\nabla^{v}(x),\]
the connection form of $\nabla$, then $R_{\nabla_{D}}=v^{*}[D,p][D,p]v+\omega_{D}^{2}+v^{*}\pi_{D}(\delta(v \omega v^{*}))v$.
\end{prop}
\begin{proof}
Theorem \ref{thm: Ristwoformdetermineduptojunk} with $S=0$ gives us that $R_{\nabla_{D}}(x)\in X\otimes^{h}_{B}\Omega^{2}_{D}$. 
By 
$\mathcal{B}_{2}$-linearity of $R_{\nabla_{D}}$ we then find
\begin{align*}
R_{\nabla_{D}}(x)=\sum_{i}R_{\nabla_{D}}(x_{i})\langle x_{i},x\rangle,
\end{align*}
and since there are finitely many elements $R_{\nabla_{D}}(x_{i})\in X\otimes_B\Omega^{2}_{D}$,
it follows that $R_{\nabla_{D}}$ is a bounded operator. Now we have 
\[1\otimes_{\nabla}D=v^{*}Dv+\omega_{D},\quad 1\otimes_{\nabla}D^{2}=v^{*}D^{2}v+\omega_{D^{2}},\]
and
\[
(1\otimes_{\nabla}D)^{2}=(v^{*}Dv+\omega_{D})^{2}=v^{*}Dvv^{*}Dv+v^{*}[{D},v\omega_{D}v^{*}]v+ \omega_{D}^{2}.
\]
Using $p[\D,p]p=0$ we compute $v^{*}Dvv^{*}Dv-v^{*}D^{2}v=v^{*}[D,p][D,p]v$ where $p=vv^{*}$, and 
\[
v^{*}[{D},v\omega_{D}v^{*}]v-\omega_{D^{2}}=v^{*}([{D},\pi_{D}(v\omega v^{*})]-\pi_{D^{2}}(v\omega v^{*}))v=v^{*}\pi_{D}(\delta(v\omega v^{*}))v ,
\] 
by Proposition \ref{ajunkie}. Thus it now follows that 
\[
R_{\nabla_{D}}=v^{*}[D,p][D,p]v+v^{*}\pi_{D}(\delta(v\omega v^{*}))v+ \omega_{D}^{2},
\]
as claimed.
\end{proof}
Although the curvature operator of a finitely generated projective module is bounded, there is no uniform bound on its norm, in the following sense. 
As an illustrative example, consider the module $\mathcal{L}_{1}$ of sections of the tautological line bundle $L_1\to \mathbb{P}^{1}(\mathbb{C})$ over the two sphere and define $\mathcal{L}^{n}:=\mathcal{L}_{1}^{\otimes^{n}}$. Then the calculations in \cite[Section 6]{BMS} show that $\|R_{\nabla^{n}}\|\sim n$ for a natural family of Grassmann connections $\nabla^{n}$.
The infinite direct sum $\bigoplus_{n=1}^{\infty} \mathcal{L}_{n}$ can be given the structure of a $C^{1}$-module, whose curvature operator is unbounded.

\section{Grassmann connections}
\label{sec:graaaaaaaaaa}
The Kasparov stabilisation theorem shows that every countably generated
Hilbert module is a complemented submodule of the standard module and thus admits a frame.
Modulo differentiability, every isometric inclusion in the standard
module yields a connection, called a Grassmann connection. 
In purely noncommutative settings, one often has access to a frame but not necessarily much else. In this section 
we provide sufficient differentiability conditions on frames and modules for $C^{2}$-Grassmann and their curvature to exist.
A range of examples comes from Cuntz-Pimsner algebras \cite{BMS, GM, GMR, Seniorthesis} including the 
$\theta$ and $q$-deformed $3$-spheres, as well as abstract constructions of connections 
in $KK$-theory \cite{K, MR}. Interestingly, we will see in Section \ref{sect:riem-subm} that our
sufficient conditions are met for Riemannian submersions as well.


To handle $\Z_2$-graded modules we need a $\Z_2$-graded standard module, which we take to be (the completion of) $H_B=H\ox B=\ell^2(\hat{\mathbb{Z}})\ox B$. Here 
$\hat{\mathbb{Z}}=\Z\setminus\{0\}$, and the $\pm$-homogenous subspaces are those indexed by positive $n\in\hat{\mathbb{Z}}$ and 
those indexed by negative $n\in\hat{\mathbb{Z}}$. A basis $\{e_i\}_{i\in\hat{\mathbb{Z}}}$ 
for $H_B$ is
homogenous if $e_i$ has positive degree if and only if $i$ is positive.

Recall that a \emph{countable frame} for a Hilbert module
$X$ is a sequence $\{x_i\}\subset X$ such that for all $x\in X$, $x=\sum_{i}x_{i}\langle x_{i},x\rangle$ as a norm convergent series. That is,  the sum ${\rm Id}_X=\sum_i |x_i\rangle\langle x_i|$ converges strictly.

\subsection{Differentiable stabilisation}
Recall from Definition \ref{diffconn} that for $k=1,\,2$, 
a horizontally $C^{k}$-submodule is a $\mathcal{B}_{k}$-submodule $\mathcal{X}\subset X$ for which $\langle x,y\rangle\in\mathcal{B}_{k}$.

\begin{defn} 
\label{diffframe}
Let $k=1,2$ and $X$ a $C^{*}$-module over $B$ and $\mathcal{X}\subset X$ a horizontally $C^{k}$-submodule. An even stabilisation isometry $v:X\to H_{B}$  \emph{horizontally $C^{k}$-differentiable} with respect to $(\mathcal{B},Y, T)$ if 
\begin{enumerate}
\item $v:X\to H_{B}$ restricts to a map $v:\mathcal{X}\to H\otimes^{h}\mathcal{B}_{k}$; 
\item there is a dense graded subspace $\mathcal{H}\subset H$ such that  $v^{*}:H_{B}\to X$ restricts to an even map $v^{*}:\mathcal{H}\otimes^{\alg}\mathcal{B}_{k}\to \mathcal{X}$.
\end{enumerate}
Given a homogenous orthonormal basis $\{e_i\}\subset\H\subset H$,
we say that the frame $\{x_i=v^*(e_i\ox1)\}$ is a $C^k$-frame.
\end{defn}
That  $\{x_i=v^*(e_i\ox1)\}$ is a frame for the 
Hilbert $C^{*}$-module $X$ is a short computation. 
\begin{prop}
\label{Ckframechoice}
Let $(\mathcal{B},Y,T)$ be a $C^{k}$ unbounded Kasparov module, $X$ a $C^{*}$-$B$-module, $\mathcal{X}$ a horizontal $C^{k}$-submodule and $v:X\to H\otimes^{h}\mathcal{B}_{k}$ a $C^{k}$-stabilisation. For any countable homogenous basis $\{e_{i}\}\subset\mathcal{H}$ for the Hilbert space $H$, the elements $x_{i}:=v^{*}(e_i\otimes 1)\in \mathcal{X}$ form a frame for $X$ with the property that for all $x\in\mathcal{X}$ the series
\[
\sum_{i\in\hat{\mathbb{Z}}}[T,\langle x_{i},x\rangle]^{*}[T,\langle x_{i},x\rangle],\quad \sum_{i\in\hat{\mathbb{Z}}}(T-i)^{-1}[T^{2},\langle x_{i},x\rangle]^{*}[T^{2},\langle x_{i},x\rangle](T+i)^{-1},
\]
are norm convergent in $\End^{*}_{C}(Y)$ (if $k=1$, only the first series converges).
\end{prop}

\begin{proof}
Let $\{e_{i}\}\subset\mathcal{H}$ be any countable homogenous basis for the Hilbert space $H$. 
Then $x_{i}:=v^{*}(e_i\otimes 1)\in \mathcal{X}$ by condition 2 of Definition \ref{diffframe}. 
 Since
\[
(\langle x_{i},x\rangle)_{i\in\hat{\mathbb{Z}}}=(\langle v^{*}(e_i\otimes 1),x\rangle)_{i\in\hat{\mathbb{Z}}}=(\langle e_i\otimes 1, v(x)\rangle)_{i\in\hat{\mathbb{Z}}},
\]
the column $(\langle x_{i},x\rangle)_{i\in\hat{\mathbb{Z}}}$ satisfies
\begin{align*}
\left\|\sum_{|i|\geq n}[T,\langle x_{i},x\rangle]^{*}[T,\langle x_{i},x\rangle]\right\| 
&= \left\|\sum_{|i|\geq n} [T,\langle e_{i}\otimes 1,v(x)\rangle]^{*}[T,\langle e_{i}\otimes 1, v(x)\rangle]\right\|\\
&= \left\|\sum_{|i|\geq n} \pi^{1}_{T}(\langle e_{i}\otimes 1, v(x)\rangle)^{*}\pi_{T}^{1}(\langle e_{i}\otimes 1, v(x)\rangle)\right\|\to 0,
\end{align*}
as claimed. To prove the norm convergence of
\[\sum_{i\in\hat{\mathbb{Z}}}(T-i)^{-1}[T^{2},\langle x_{i},x\rangle]^{*}[T^{2},\langle x_{i},x\rangle](T+i)^{-1},\]
we use the same argument, estimating with the representation $\pi^{2}_{T}$ of Equation \eqref{diffreps}.
\end{proof}
\subsection{$C^{2}$-Grassmann connections}
\begin{lemma}

Let $\mathcal{X}$ be a horizontally $C^{1}$-module for $(\mathcal{B},Y,T)$ and $(v,H_B)$ a $C^{1}$-stabilisation. 
The Grassmann connection
\begin{equation*}
\nabla^{v}:\mathcal{X}\to 
X\otimes_{B}^{h}\Omega^{1}_{u}(\mathcal{B}_{1},B),
\quad \nabla^{v}(x):=(\gamma(v)^{*}\otimes 1)(1\otimes\delta)v(x),
\end{equation*}
is defined on $\X$. For any homogenous orthonormal basis $\{e_i\}\subset\H\subset H$, we can use the  $C^{1}$-frame 
$\{x_i=v^*(e_i\ox 1)\}_{i\in\hat{\mathbb{Z}}}$ 
to express the Grassmann connection as
\begin{equation}
\label{Gra}\nabla^{v}(x)=\sum_{i\in\hat{\mathbb{Z}}} \gamma(x_{i})\otimes\delta(\langle x_{i},x\rangle),\end{equation}
as a norm convergent series.  Consequently 

\[
\nabla^{v}_{T}
:=\pi_{T}(\nabla^{v}):\mathcal{X}\to X\otimes_{B}^{h}\Omega^{1}_{T}(\B_1),\quad \nabla_{T}^{v}(x)=\sum_{i\in\hat{\mathbb{Z}}} \gamma(x_{i})\otimes [T,\langle x_{i},x\rangle],
\]
is well-defined and independent of the choice of orthonormal basis in $\H\subset H$.
\end{lemma}
\begin{proof}
The derivation $\delta:\mathcal{B}_{1}\to \Omega^{1}(\mathcal{B}_{1},B)$ is completely contractive, hence 
\[
1\otimes\delta:H\otimes^{h}\mathcal{B}_{1}\to H\otimes^{h}\Omega^{1}(\mathcal{B}_{1},B)\xrightarrow{\sim}H\otimes^{h}\mathcal{B}_{1}\otimes_{\mathcal{B}_{1}}^{h}\Omega^{1}(\mathcal{B}_{1},B),
\]
 is defined. Thus the composition $\nabla^{v}:=(v^{*}\otimes 1)(1\otimes\delta)v$ is defined on $\mathcal{X}$. Choose 
 an orthonormal basis $\{e_i\}\subset\H\subset H$
 and form the $C^{1}$-frame $\{x_{i}=v^{*}(e_{i}\otimes 1)\}_{i\in\hat{\mathbb{Z}}}$.
 Then
\begin{align*}
\nabla^{v}(x)&=(\gamma(v)^{*}\otimes 1)(1\otimes\delta)v(x)=\gamma(v)^{*}(1\otimes\delta)\left(\sum_{i\in\hat{\mathbb{Z}}} e_{i}\otimes \langle e_{i}\otimes 1,v(x)\rangle\right)\\
&=\gamma(v)^{*}(1\otimes\delta)\left(\sum_{i\in\hat{\mathbb{Z}}} e_{i}\otimes \langle x_{i},x\rangle\right)=\gamma(v)^{*}\left(\sum_{i\in\hat{\mathbb{Z}}} e_{i}\otimes \delta(\langle x_{i},x\rangle)\right)\\
&=\sum_{i\in\hat{\mathbb{Z}}} \gamma(x_{i})\otimes \delta(\langle x_{i},x\rangle).
\end{align*}
Hence the represented connection 
$\nabla_{T}^{v}=(1\otimes \pi_{T})\circ\nabla^{v}:\mathcal{X}\to 
X\otimes_{B}^{h}\Omega^{1}_{T}(\B_1)$ is well-defined and independent of the choice of orthonormal basis.
\end{proof}
By Lemma \ref{symmetric} 
we obtain a densely-defined symmetric operator
\[
1\otimes_{\nabla^{v}}T:\mathcal{X}\otimes^{\alg}_{\mathcal{B}}\Dom T
\to X\otimes_{B}H, \quad x\otimes h\mapsto \gamma(x)\otimes Th+\nabla_{T}^{v}(x)h,
\]
in the Hilbert $C^{*}$-module $X\otimes_{B}Y$.
\begin{prop} 
\label{C2frameconnection}
Let  $(\mathcal{B},Y,T)$ be a $C^{k}$-Kasparov module, $(\mathcal{X},S)$ a $C^{1}$-module if $k=1$ and a $C^{(1,2)}$-module if $k=2$,
and $(v,\mathcal{H})$ a $C^{k}$-stabilisation.
Suppose that
\begin{equation}
\label{quasigrass}
(\gamma\otimes T)(v\otimes 1)-(v\otimes 1)(1\otimes_{\nabla^v} T):\mathcal{X}\otimes^{\alg}_{\mathcal{B}_{1}}\Dom T\to H_{Y},
\end{equation}
extends to $\Dom S\otimes 1$.  If  $(S\otimes 1, 1\otimes_{\nabla^{v}}T)$ is a vertically anticommuting pair, then
\[
\nabla^{v,S}:\mathcal{X}\to\mathcal{X}_{\nabla^{v}_{T}}^{S}\otimes^{h}_{\mathcal{B}_{1}}\Omega^{1}_{u}(\mathcal{B}_{1},\mathcal{B}_{2}),\quad  \nabla^{v,S}(x):=(\gamma(v)^{*}\otimes 1)(1\otimes\delta)v(x),
\]
is well-defined and the pair $(\nabla^{v},\nabla^{v,S})$ defines a $C^{2}$-connection on $(\mathcal{X},S)$. 
\end{prop}
\begin{proof} 
The universal differential 
$\delta:\mathcal{B}_{k}\to \Omega^{1}_{u}(\mathcal{B}_{1},\mathcal{B}_{k})$ 
is given by $\delta (a)=1\otimes a-a\otimes 1$. Choose an
orthonormal basis $\{e_i\}\subset\H\subset H$ and form the 
$C^{k}$-frame $\{x_{i}=v^*(e_i\ox1)\}_{i\in\hat{\mathbb{Z}}}$.
As a $C^{k}$-frame is in particular a $C^{1}$-frame, by Equation \eqref{Gra} we have
\begin{align*}
\nabla^{v}(x)=\sum_{i\in\hat{\mathbb{Z}}} \gamma(x_{i})\otimes \delta(\langle x_{i},x\rangle).
\end{align*}
It suffices to show that this series is convergent in the Haagerup norm of 
the tensor product $\mathcal{X}_{\nabla_{T}^{v}}^{S}\otimes^{h}_{\mathcal{B}_{1}}\Omega^{1}_{u}(\mathcal{B}_{1},\mathcal{B}_{k})$, for then we can define
\[\nabla^{v,S}(x)=\sum \gamma(x_{i})\otimes \delta(\langle x_{i},x\rangle)\in \mathcal{X}_{\nabla_{T}^{v}}^{S}\otimes^{h}_{\mathcal{B}_{1}}\Omega^{1}_{u}(\mathcal{B}_{1},\mathcal{B}_{k}).\]
In order to prove norm-convergence of $\nabla^{v,S}(x)$ 
using the Haagerup tensor norm, by part 2 of 
Proposition \ref{Haaconvergent}, we need to address summability of the column with entries
\[
\delta(\langle x_{i},x\rangle)
=1\otimes \langle x_i,x\rangle-\langle x_i,x\rangle\otimes 1,
\]
as well as boundedness of the row $(\gamma(x_{i}))$ in 
$\mathcal{X}_{\nabla_{T}^{v}}^{S}$.
Now $C^{k}$-column finiteness guarantees 
that this column is in 
$H_{\Omega^{1}_{u}(\mathcal{B}_{1},\mathcal{B}_{k})}$. 
It thus remains to show that the row $(\gamma(x_{i}))$ is bounded in 
$\mathcal{X}_{\nabla^{v}_{T}}^{S}$.

By Equation \ref{quasigrass}, $(1-vv^{*})(\gamma\otimes T)v(S+i)^{-1}$ is a 
bounded operator. Computing the norm of the 
row $\gamma(x_{i})$ in $\mathcal{X}_{\nabla_{T}^{v}}^{S}$ gives
\begin{align*}
\left\| \left(|\gamma(x_{i})\rangle\right)_{i}^{t}\right\|_{\mathcal{X}^{S}_{\nabla_{T}^{v}}}
\leq \Big\| \left(|(S+i)^{-1}\nabla_{T}^{v}(\gamma(x_{i}))\rangle\right)_{i}^{t}\Big\|
+\left\| \left(|\gamma(x_{i})\rangle\right)_{i}^{t}\right\|.
\end{align*}
The second term has norm $1$ 
and the first term is estimated by 
\begin{align*}
\Big\| \left(|(S+i)^{-1}\nabla_{T}^{v}(\gamma(x_{i}))\rangle\right)_{i}^{t}\Big\| 
&=\Big\|\left((S+i)^{-1}\gamma(v)^{*}[\gamma\otimes T,v(\gamma(x_{i}))]\right)^{t}_{i}\Big\|_{\mathbb{B}(H_Y, X\otimes_{B} Y)}\\
&=\left\|\left( [\gamma\otimes T,\langle x_{i}| v^{*}] v(S-i)^{-1}\right)_{i}\right\|_{\mathbb{B}( X\otimes_{B} Y, H_{Y})}\\
&=\left\|\left( (\gamma\otimes T \langle x_{i}| -\langle x_{i}|1\otimes_{\nabla^{v}}T)
(S-i)^{-1}\right)_{i}\right\|_{\mathbb{B}( X\otimes_{B} Y, H_{Y})}\\
&\leq \left\|\left( ((1-vv^{*})(\gamma\otimes T)v)(S-i)^{-1}\right)\right\|,
\end{align*}
which remains bounded since Equation \eqref{quasigrass}
tells us that $((1-vv^{*})\otimes 1)(\gamma\otimes T)(v\otimes 1)$ extends to $\Dom (S\otimes 1)$.
\end{proof}
\begin{rmk}
In \cite[Lemma 4.2, Theorem 4.7]{K} it is proved that, given $X$ and $(\B,Y,T)$, one can find 
a dense $\mathcal{B}_{1}$-submodule $\mathcal{X}\subset X$ and an isometry $v:X\to \ell^{2}(\mathbb{Z})\otimes^{h} B$ 
such that $v$ restricts to a map $v:\mathcal{X}\to  \ell^{2}(\mathbb{Z})\otimes^{h} \mathcal{B}_{1}$ and $v^{*}$ 
restricts to a map $v^{*}:C_{c}(\mathbb{Z})\otimes^{\alg}\mathcal{B}_{1}\to \mathcal{X}$. Furthermore, using \cite[Theorem 3.9]{K}
it can be shown that $(v\otimes 1)(1\otimes_{\nabla^{v}}T)-(\gamma\otimes T)(v\otimes 1)$ is defined on the range of a certain explicit positive compact
operator $K$. The inverse of $K$, made odd in an appropriate way, is a natural candidate for a vertical operator $S$. It is unclear however
whether such $S\otimes 1$ and $1\otimes_{\nabla_{v}}T$ can be made to vertically anticommute, so that Proposition \ref{C2frameconnection} can be applied.
This is subject of future research.
\end{rmk}
\section{The curvature operator of a Riemannian submersion}
\label{sect:riem-subm}

We will now illustrate our notion of curvature for a large class of 
examples given by Riemannian submersions of 
closed spin$^c$ manifolds $M\to B$
that were analysed using techniques from unbounded 
$KK$-theory in \cite{KS16}. The main result therein (\cite[Theorem 23]{KS16}) was a factorisation of essentially self-adjoint operators of the form
\begin{equation}
  \label{eq:factorization}
D_M = D_V \otimes 1 + 1 \otimes_\nabla D_B + \tilde c (\Omega).
\end{equation}
where $D_M$ is the Dirac operator on the total space, $D_V$ a vertical family of Dirac operators, $D_B$ the Dirac operator on the base manifold lifted to an operator $1\ox_\nabla D_B$ on $M$ using the connection $\nabla$ and, finally, $\tilde c(\Omega)$ is (Clifford multiplication by) the curvature of the Riemannian submersion ({\em cf.} Definition \ref{defn:riem-subm-curv} below).

As we will see, in this case the curvature operator of Definition \ref{def: curv-operator} ---for which we will use the short-hand $R_{(D_V,\nabla)}:= R_{(D_V,\nabla_{D_B})}$---indeed captures curvature of the connection $\nabla$ on the vertical Hilbert module of the submersion, as well as other geometric information such as the mean curvature. We will check that the conditions that enter in our general framework are indeed fulfilled in this concrete geometric context. But first we give a summary of the geometric setup. 
\subsection{Geometric setup}
Let us start by recalling from \cite{KS16} the relevant ingredients, refering to that paper for all details. Thus, we consider a Riemannian submersion of closed Riemannian manifolds $\pi: M \to B$. Recall the following tensors that are associated to this structure.
\begin{defn}
  \label{defn:riem-subm-curv}
\begin{enumerate}
\item The {\em second fundamental form}, defined for real vertical vector fields $X,Y$ and real horizontal vector fields $Z$ on $M$ by 
\begin{equation}
\label{eq:2nd-fund-form}
S_\pi(X,Y,Z) := \frac{1}{2}\big( Z (\langle X,Y\rangle_M) - \langle[Z,X],Y\rangle_M - \langle[Z,Y],X\rangle_M \big)
\end{equation}
\item The \emph{mean curvature} $k \in \pi^* \Omega^1(B)$ is given as the trace
\[
k = (\tr \otimes 1) (S_\pi) \, .
\]
\item The {\em curvature of the fibre bundle $\pi : M \to B$} is given by the element $\Omega$ in $\Omega^2(M) \otimes_{C^\infty(M)} \Omega^1(M)$ given by 
\[
\Omega(X,Y,Z) := -\langle [ (1-P)X,(1-P)Y], PZ\rangle_M
\]
where $P$ is the orthogonal projection onto vertical vector fields. 
\end{enumerate}
\end{defn}

If $M$ and $B$ are Riemannian spin manifolds, we may introduce a vertical spinor module $\E_V$, defined in terms of the spinor modules $\E_M$ and $\E_B$ on the given spin$^c$ manifolds $M$ and $B$, respectively \cite[Section 3]{KS16}. We will not dwell on the precise definition here, but merely recall that $\E_V$ is a finitely-generated projective $C^\infty(M)$-module which satisfies the crucial property that $\E_V \otimes \pi^* \E_B \cong \E_M$ where $\pi^* \E_B \equiv \E_B \otimes_{C^\infty(B)} C^\infty(M)$ denotes the pullback. Moreover, Clifford multiplication $c_V$ by vertical vector fields is defined on $\E_V$.

The module $\E_V$ has a (Clifford) connection $\nabla^{\E_V}$ defined  
in terms of the spinor connections $\nabla^{\E_M}$ 
and $\nabla^{\E_B}$. The connection $\nabla^{\E_V}$ 
is Hermitian for the natural Hermitian structure 
$\langle\cdot,\cdot \rangle_{\E_V}$ on $\E_V$. 
The smooth sections of $\E_V$ have a natural locally convex 
structure coming from the usual $C^\infty$-topology, which can be defined
using $\nabla^{\E_V}$.


We now define the pre-Hilbert module $\X$ over $C^\infty(B)$ to be $\E_V$ where the 
right action of $C^\infty(B)$ is defined via the inclusion
$C^\infty(B)\to C^\infty(M)$ dual to $\pi:\,M\to B$. The $C^\infty(B)$-valued
inner product is defined by
\begin{equation}
\langle s,t \rangle_X (b):= \int_{\pi^{-1}(b)} \langle s,t \rangle_{\E_V} (x)\,d\mu_{\pi^{-1}(b)}(x),\quad s,\,t\in \E_V. 
\label{eq:ex-inn-prod}
\end{equation}
Here $d\mu_{\pi^{-1}(b)}$ is the Riemannian volume form on the submanifold $\pi^{-1}(b)$. As $b\mapsto d\mu_{\pi^{-1}(b)}$ is smooth (the volume form 
on $M$ decomposes locally as a product),
this inner product does in fact take values in $C^{\infty}(B)$.
\begin{prop}{\cite[Proposition 16]{KS16}}
  Let $\{e_j\}$ be a local orthonormal frame of vertical vector fields on $M$. Then the following local expression defines an odd symmetric unbounded operator $(D_V)_0 : \X \to X$:
\[
(D_V)_0(\xi) = i  \sum_{j=1}^{\dim (F)} c_V(e_j) \nabla^{\E_V}_{e_j}(\xi) 
\]
The closure $D_V: {\rm dom}(D_V) \to X$ of $(D_V)_0$ is regular and self-adjoint.
\end{prop}

In order to form the unbounded Kasparov product of the vertical and the horizontal components we need to lift the Dirac operator $D_B$ on the base manifold to an essentially self-adjoint unbounded operator on the Hilbert space $X \otimes_{C(B)} L^2(\E_B)$. It turns out that the Hermitian Clifford connection $\nabla^{\E_V}$ on $\E_V$ (Hermitian with respect to the $C^\infty(M)$-valued inner product) does \emph{not define} a metric connection on $\E_V \subseteq X$ with the $C^\infty(B)$-valued inner product, due to correction terms that come from the measure on the fibres $M_b$, $b \in B$. 
However, these can be nicely absorbed in an additional term proportional to the 
mean curvature.
\begin{defn}
  \label{defn:metric-conn}
  The metric connection $\nabla^{\X} : \X \to X \otimes_{C(B)} \Omega^1_{\Tex{cont}}(B)$ is defined by
\[
\nabla_Z^{\X} (\xi) = \nabla_{Z_H}^{\E_V} (\xi) + \frac{1}{2} k(Z_H) \cdot \xi
\]
in terms of a vector field $Z$ on $B$ and corresponding horizontal lift $Z_H$. 
\end{defn}
\begin{prop}
  Let $\{f_i\}$ be a local orthonormal frame of vector fields on $B$. The local expression 
  \[
  \begin{split}
& (1 \otimes_\nabla D_B)_0 (\xi \otimes r)  := \xi \otimes D_B( r) + i \sum_i \nabla^\X_{f_i} (\xi)  \otimes c_B(f_i)(r) \qquad
\xi \in \E_V \, , \, \, r \in \E_B
\end{split}
\]
defines an essentially self-adjoint unbounded operator 
\[
(1 \otimes_\nabla D_B)_0 : \E_V \otimes_{C^\infty(B)}^{\rm alg} \E_{B} \to X \otimes_{C(B)} L^2(\E_{B}) \, .
\]
We denote its closure by $1\otimes_\nabla D_B : {\rm Dom}(1\otimes_\nabla D_B ) \to X \otimes_{C(B)} L^2(\E_{B})$.
\end{prop}

\begin{rmk}
In \cite{KS16} it is only shown that $(1 \otimes_\nabla D_B)_0$ is a symmetric operator, whose closure was then used in the tensor sum factorization \eqref{eq:factorization} of $D_M$. However, we may consider the operator $(1 \otimes_\nabla D_B)_0$ as a differential operator of order $1$ on the finitely-generated projective $C^\infty(M)$-module $\E_V \otimes_{C^\infty(M)} \pi^*\E_B$. It is then a classical result \cite[Corollary 10.2.6]{HR} that such an operator is essentially self-adjoint. 
\end{rmk}

The main result of \cite{KS16} is a factorisation of the Dirac operator $D_M$ on $M$ in terms of a vertical family of Dirac operators $D_V$ on $X$ and the Dirac operator $D_B$ on $B$. Explicitly, (up to conjugation by a unitary operator) $D_M$ is given by the tensor sum \eqref{eq:factorization}:
\begin{equation*}
D_M = D_V \otimes 1 + 1 \otimes_\nabla D_B + \tilde c (\Omega).
\end{equation*}
The last term $\tilde{c}(\Omega)$ is Clifford multiplication  by 
the curvature $\Omega$ of the fibration $\pi: M \to B$. Thus the curvature of the fibration appears as an obstruction 
to the realisation of $D_M$ as an (unbounded) internal 
Kasparov product. We also record the following result, 
which is Lemma 17 of \cite{KS18}. 

\begin{lemma}
\label{lma:commutator-Kuc}
Suppose that $\xi \in \X $ and $r \in \E_B \subseteq L^2(\E_B)$. 
Let $\{e_j\}$ denote a local orthonormal frame of vertical 
vector fields on $M$ and $\{f_i\}$ a local 
orthonormal frame of vector fields on $B$. 
Then we have the local expression
\begin{align*}
&[D_V \otimes 1, 1 \otimes_\nabla D_B]_+(\xi \otimes r)\\  \nonumber 
&\quad = -\sum_{i,j,k} S_\pi (e_k, e_j, (f_i)_H)  \big(c_V(e_j) \nabla_{e_k}^{\E_V} \big) (\xi) \otimes  c_B(f_i)(r)\\ \nonumber 
& \qquad - \sum_{i,j} c_V(e_j)\Big( \Omega^{\E_V}( e_j, (f_i)_H)+\frac 12  e_j \big( k((f_i)_H) \big) \Big)(\xi)\otimes  c_B(f_i)( r) \, ,
\end{align*}
where 
$\Omega^{\E_V} : \E_V \to \E_V \otimes_{C^\infty(M)} \Omega^2(M)$ 
is the curvature form of the Hermitian connection $\nabla^{\E_V}$. 
\end{lemma}

\color{black}

As a consequence we obtain that the anti-commutator 
$[D_V \otimes 1, 1 \otimes_{\nabla} D_B]_+$ is relatively bounded by 
$D_V \otimes 1$ and makes $D_V\otimes 1$ and $1 \otimes_\nabla D_B$ 
a vertically anticommuting pair in the sense of Definition \ref{wac}. 
\subsection{Local expressions}
Let us first take a typical fibre $F_0$ of $\pi:M \to B$ 
and consider trivialisations
$$
\rho_\alpha : \pi^{-1}(W_\alpha) \to F_0 \times W_\alpha
$$
for charts $W_\alpha \subseteq B$. It is then possible to choose so-called {\em fibration charts} $V_{\alpha} \times W_{\alpha} \subset F_0 \times B$ so that with $U_\alpha = \rho^{-1}_\alpha(V_{\alpha} \times W_{\alpha})$ we have that
$$
\rho_\alpha: U_\alpha \to V_{\alpha} \times W_{\alpha}
$$
is a diffeomorphism. Let us denote by $\{\chi_\alpha^2\}$ 
a partition of unity subordinate to the covering $\{U_\alpha\}$, 
so that $\sum_\alpha \chi_\alpha^2 =1$.
Note that $\E_V$ 
as a (finitely-generated projective) $C^\infty(M)$-module admits 
a (finite) local orthonomal frame supported on $U_\alpha$; we denote such a frame by $\{x_{\alpha,n}\}$.
\color{black}



On the base manifold $B$ we choose local coordinates 
$\sigma_{\alpha}: W_{\alpha} \to \R^{\dim B}$ whose 
components will be denoted by $\sigma_\alpha^\mu : W_\alpha \to \R$ 
for $\mu=1,\ldots, \dim B$. 

It is convenient to write the inner product on $X$ as an integral over the typical fibre $F_0$. We denote by $\mu_b$ the measure on $F_0$ that corresponds to $\mu_{\pi^{-1}(b)}$ through the identification $\pi^{-1}(b) \cong F_0$. We then obtain 
$$
\langle s,t \rangle_X (b) =\sum_\alpha \int_{F_0}\chi_\alpha^2 \left((\rho_\alpha^{-1})^* \langle s,t\rangle_{\E_V} \right) (y,b) \, d\mu_b(y)
$$
for all $b \in W_{\alpha}$. 
As a special case, if $(y,b_0) \in \rho_{\alpha_0} (U_{\alpha_0})$ where $\pi ^{-1} (b_0) =F_0$ is our typical fibre then we set $d\mu_0:= d\mu_{b_0}$.

The Radon-Nikodym derivative of $\mu_b$ with respect to $\mu_0$ gives a function on $F_0 \times B$ which we will denote by
$$
\frac {d\mu }{d\mu_0} (y,b) := \frac{d\mu_b}{d\mu_0}(y)
$$

\begin{rmk}
  As in \cite{KS16} we may combine a choice of coordinates on each of the fibration charts $\{ U_\alpha\}$ with the Riemannian metric $g$
  on $M$ to obtain a positive invertible matrix of smooth functions
\[
g_\alpha : U_\alpha \to \Tex{GL}_{\dim(M)}(\R)_+.
\]
Furthermore, letting $Q : \R^{\dim(M)} \to \R^{\dim(M)}$ denote the projection 
$$
Q : (t_1,\ldots,t_{\Tex{dim}(M)}) \mapsto (t_1, \ldots,t_{\dim (F)}, 0,\dots, 0)
$$ 
onto the first $\dim(F)$ copies of $\R$ in  $\R^{\dim(M)}$, we obtain a 
positive matrix of smooth functions
\[
Q g_\alpha Q : U_\alpha \to \Tex{GL}_{\dim(F)}(\R)_+,\qquad Q g_\alpha Q(x)=Q g_\alpha(x) Q.
\]
Suppose that $(y,b) \in \rho_{\alpha}(U_\alpha)$ and let $y^1_\alpha, \ldots, y^{\dim F_0}_\alpha$ denote local coordinates on $V_\alpha\subset F_0$. Then we may write the volume form on $F_0$ as
$$
d\mu_b (y)= \sqrt{ \det Qg_\alpha Q (\rho_\alpha^{-1}(y,b))} dy_\alpha^1 \wedge \cdots \wedge dy_\alpha^{\dim F_0}. 
$$
Let us check, for completeness, that this expression does not depend on the choice of trivialisation. In fact, if $(y,b) \in \rho_\alpha (U_\alpha)\cap \rho_\beta(U_\beta)$ then the transition functions $\rho_\alpha \circ \rho_\beta^{-1}$ map $(y,b)$ to $(y',b)$. In  terms of the coordinates $y_\alpha^k$ and $y_\beta^l$, the map $y \to y'$ corresponds to an orthogonal transformation $T^{kl}(y) := \partial y_\alpha^k / \partial y_\beta^l$ in $\R^{\dim (F)}$. Hence we have 
$$
\det Qg_\beta Q (\rho_\beta^{-1}(y,b)) =  \det  T^{2}\cdot  \det Qg_\alpha Q (\rho_\alpha^{-1}( y,b))
$$
while at the same time
$$
dy_\beta^1 \wedge \cdots \wedge dy_\beta^{\dim F_0}  =  \det \frac{\partial y_\beta^k}{\partial y_\alpha^l} \cdot  dy_\alpha^1 \wedge \cdots \wedge dy_\alpha^{\dim F_0} = \det T^{-1} \cdot  dy_\alpha^1 \wedge \cdots \wedge dy_\alpha^{\dim F_0}
$$
so that the terms involving $\det T$ cancel in the definition of $d\mu_b$. 

The Radon-Nikodym derivative can now be written unambiguously on $V_\alpha \times W_\alpha$ by
$$
\frac {d\mu }{d\mu_0} (y,b) = \frac{\sqrt{\det Qg_\alpha Q (\rho^{-1}_\alpha(y,b))}} {\sqrt{\det Qg_{\alpha_0}Q (\rho^{-1}_{\alpha_0}(y,b_0))}} .
$$
This is a smooth and nowhere vanishing function on $V_\alpha \times W_\alpha$.
  \end{rmk}

\subsection{The stabilisation isometry}
We are now ready to define the crucial technical ingredient in our approach to curvature, to wit, a stabilising isometry $v: X \to L^2(F_0)^{N} \otimes^h C(B)$. In fact, we will realise this map on $\X$ where it will map to $C^\infty(F_0)^N \widehat\otimes C^\infty(B)$ (in terms of the projective tensor product of Fr\'echet spaces). We let $N$ denote the product of the cardinalities of the sets $\{ \chi_\alpha\} $ and $\{x_{\alpha,n}\}$ for the partition of unity and the frame of $\E_V$, respectively, and will consider elements in $L^2(F_0)^N$ as column vectors.

\begin{lemma}
The map 
  \begin{align*}
  v: \X &\to C^\infty(F_0)^N \widehat\otimes C^\infty(B)\simeq C^\infty(F_0\times B)^{N} \\ 
s &\mapsto \left(\sqrt{\frac{d\mu}{d\mu_0}} \cdot (\rho_{\alpha}^{-1})^* (\langle \chi_\alpha x_{\alpha,n} , s\rangle_{\E_V} ) \right)_{\alpha n}
  \end{align*}
  is a continuous map of Fr\'echet spaces and furthermore  extends to an isometry $X \to L^2(F_0)^N \otimes^h C(B)$.
\end{lemma}

\proof
First observe that as each $\rho_\alpha$ is a diffeomorphism, 
the map $\rho_\alpha^*$ is continuous in the $C^\infty$-topology.
Likewise the derivatives $\sqrt{d\mu/d\mu_0}$ are (uniformly bounded) $C^\infty$ functions,
and so multiplication by them is continuous. Now for every section
$s$ we have 
$$
s=\sum_\alpha\chi_\alpha^2s=\sum_{\alpha,n}\chi_\alpha x_{\alpha,n}\langle \chi_\alpha x_{\alpha,n},s\rangle,
$$
and each term in the sum depends continuously on $s$. 
Taking the inner product with $\chi_\beta x_{\beta,k}$ is 
$C^\infty$ continuous, as is $(\rho_\alpha^{-1})^*$.
Hence the map $v$ is Fr\'{e}chet continuous.

By a  standard density argument , to show that $v$ extends to an isomtery, it is sufficient to check that for all $s,t \in \X$ we have $\sum_{\alpha n} \langle v(s)_{\alpha n},v(t)_{\alpha n} \rangle_{L^2(F_0) \otimes^h C(B)} = \langle s,t \rangle_X$. We compute
\begin{align*}
  \sum_{\alpha, n} \langle v(s)_{\alpha n}, & v(t)_{\alpha n} \rangle_{L^2(F_0) \otimes C(B)}(b) 
 = \sum_{\alpha ,n} \int_{F_0} \overline{v(s)_{\alpha n} (y,b)} v(t)_{\alpha n} (y,b)\, d\mu_0(y) \\
  & = \sum_{\alpha ,n}  \int_{F_0}  (\rho_{\alpha}^{-1})^* (\langle  s, \chi_\alpha x_{\alpha,n} \rangle_{\E_V} \langle \chi_\alpha x_{\alpha,n} , t\rangle_{\E_V} ) (y,b)\,\frac{d\mu}{d\mu_0}(y,b)\, d\mu_0(y)\\
  & = \sum_\alpha\int_{F_0} \chi_\alpha^2(\rho_{\alpha}^{-1})^*\left( \langle s,t \rangle \right) (y,b) \,d\mu_b(y)  \equiv \langle s,t\rangle_X
\end{align*}
using completeness of the frame $\{x_{\alpha,n}\}$ and  $\sum_\alpha \chi_\alpha^2 =1$ in the last equality.
\endproof

\begin{lemma}
  For all $F = (F_{\alpha n}) \in C^\infty(F_0)^N \widehat\otimes C^\infty(B)$ we have
\begin{equation}
v^*(F) = \sum_{\alpha n}  \left( (\rho_\alpha^{-1})^* \left(\sqrt{\frac{d\mu_0}{d\mu}} F_{\alpha n} \right) \chi_\alpha x_{\alpha,n}    \right)_{\alpha n}
\label{eq:vee-adjoint}
\end{equation}
Consequently, the adjoint $v^*: L^2(F_0)^N \otimes^h C(B)\to X$ restricts to a map $$v^{*}:C^\infty(F_0)^N \widehat\otimes C^\infty(B)\to \X,$$ and we have $v^*v= 1$. 
  \end{lemma}
\proof
We check that the formula \eqref{eq:vee-adjoint} does indeed
provide the adjoint of $v$ by computing
\begin{align*}
  \langle s, v^*(F)\rangle_X (b) & =
  \int_{F_0} (\rho_\alpha^{-1})^*\left(\langle s, \chi_\alpha x_{\alpha,n} \rangle_{\E_V}\right)(y,b) \, \left(\sqrt{\frac{d\mu_0}{d\mu}}F_{\alpha n}\right)(y,b) \, d\mu_b (y)\\& =
  \int_{F_0} (\rho_\alpha^{-1})^*\left(\langle s, \chi_\alpha x_{\alpha,n} \rangle_{\E_V}\right)(y,b) \, \left(\sqrt{\frac{d\mu}{d\mu_0}}F_{\alpha n}\right)(y,b) \,d\mu_0 (y)\\
  & = \langle v(s),F \rangle_{L^2(F_0)^N \otimes^h C(B)}
\end{align*}
The identity $v^*v=1$ holds true by construction. 
\endproof
Thus, the operator $v$ defines a $C^{2}$-stabilisation $v:\mathcal{X}\otimes_{C^{\infty}(B)}^{\alg}C^{2}(B) \to L^{2}(F_0)\otimes^{h} C^{2}(B)$
 in the sense of Definition \ref{diffframe}.
\begin{prop}
\label{prop: quasicomm}
The operator 
$$
\left((v\otimes 1) (1 \otimes_\nabla D_B) - (\gamma\otimes D_B) (v \otimes 1) \right)  : \X \otimes_{C^\infty(B)} ^{\alg}\E_B \to C^\infty(F_0) \widehat\otimes \E_B\subset L^{2}(F_0)\otimes^{h} C(B)
  $$
is $D_V$-bounded. 
  \end{prop}
\proof
We start by computing the first term on $s \otimes \psi \in \X \otimes \E_B$, say, with $\supp s \subseteq U_\alpha$:
\begin{align}
 (v\otimes 1 ) (1 \otimes_\nabla D_B)(s \otimes \psi)& = (v\otimes 1) (\nabla^\X_{\partial/\partial \sigma_\alpha^\mu} (s)\otimes  \gamma^\mu \psi + s \otimes D_B \psi)\\
  &  = \bigg(\sqrt{ \frac{d\mu}{d\mu_0}}  \cdot (\rho_{\alpha}^{-1})^* (\langle \chi_\alpha x_{\alpha,n} , \nabla^\X_{\partial/\partial \sigma_\alpha^\mu}s\rangle_{\E_V} )\gamma^\mu \psi \nonumber\\
  &\qquad \qquad +  \sqrt{\frac{d\mu}{d\mu_0}}  \cdot (\rho_{\alpha}^{-1})^* (\langle \chi_\alpha x_{\alpha,n} ,s \rangle_{\E_V}) \cdot D_B \psi\bigg)_{\alpha n}
  \label{eq:1st-term-vD}
\end{align}
On the other hand, we have
\begin{align}
  (D_B)_\epsilon (v \otimes 1)(s \otimes \psi) &=  \left(D_B \left(\sqrt{\frac{d\mu}{d\mu_0}}  (\rho_{\alpha}^{-1})^* (\langle \chi_\alpha x_{\alpha,n} ,s \rangle_{\E_V})  \cdot \psi \right)\right)_{\alpha n} \nonumber \\
  &=\Bigg(\frac{\partial}{\partial \sigma_\alpha^\mu} \left(\sqrt{\frac{d\mu}{d\mu_0}}\right)  (\rho_{\alpha}^{-1})^* (\langle \chi_\alpha x_{\alpha,n} ,s \rangle_{\E_V})  \cdot \gamma^\mu\psi \nonumber \\
  &\qquad  \qquad   +   \sqrt{\frac{d\mu}{d\mu_0}}  (\rho_{\alpha}^{-1})^* (\langle \chi_\alpha x_{\alpha,n} ,s \rangle_{\E_V})  \cdot D_B \psi   \Bigg)_{\alpha n}
  \label{eq:2nd-term-vD}
  \end{align}
The first term in this last expression is bounded as it is a derivative of the Radon-Nikodym derivative, while the last term cancels against the corresponding term in \eqref{eq:1st-term-vD}. We are thus left to consider
\begin{align*}
&   (\rho_{\alpha}^{-1})^* (\langle \chi_\alpha x_{\alpha,n} , \nabla^\X_{\partial/\partial \sigma_\alpha^\mu}s\rangle_{\E_V} ) - \frac{\partial}{\partial \sigma_\alpha^\mu} \left( (\rho_{\alpha}^{-1})^* (\langle \chi_\alpha x_{\alpha,n} ,s \rangle_{\E_V}) \right) \\
&   \qquad =
 - (\rho_{\alpha}^{-1})^* (\langle  \nabla^{\E_V}_{(\partial/\partial \sigma_\alpha^\mu)_H}(\chi_\alpha x_{\alpha,n}),s \rangle_{\E_V} ) + \frac 12 (\rho_{\alpha}^{-1})^*  \left( k((\partial/\partial \sigma_\alpha^\mu)_H) \langle  \chi_\alpha x_{\alpha,n},s \rangle_{\E_V} \right) \\
  & \qquad \qquad +  (\rho_{\alpha}^{-1})^* \left(\left(\frac{\partial}{\partial \sigma_\alpha^\mu}\right)_H \langle \chi_\alpha x_{\alpha,n} ,s \rangle_{\E_V}\right) - \frac{\partial}{\partial \sigma_\alpha^\mu} \left( (\rho_{\alpha}^{-1})^* (\langle \chi_\alpha x_{\alpha,n} ,s \rangle_{\E_V}) \right).
\end{align*}
The first two terms on the right-hand side (involving the derivative on the frame and the mean curvature) is bounded, and we claim that the remaining terms combine to give only vertical derivatives, and can thus be relatively bounded with respect to the vertically elliptic $D_V$ when acting on $s$. In order to see that the combination is a vertical derivative, let us consider the more general expression 
$$
(\rho_{\alpha}^{-1})^* Z_H (f) - Z ((\rho_{\alpha}^{-1})^* (f)) 
$$
for a vector field $Z$ on $B$ and a function $f$ on $M$ (supported in $U_\alpha$). Here we understand $Z$ to act on a function on $F_0 \times B$ by only deriving in the second coordinate. For $f= \pi^* g$ one finds that
$$
(\rho_{\alpha}^{-1})^* Z_H (\pi^* g ) - Z ((\rho_{\alpha}^{-1})^* (\pi^* g ))
= (\rho_{\alpha}^{-1})^*  Z_H (g \circ \pi ) - Z( g \circ \pi \circ \rho_\alpha^{-1} )=0
$$
by definition of the horizontal lift
\begin{align*}
(\rho_\alpha^{-1})^*Z_H (g \circ \pi) &= (\rho_\alpha^{-1})^*\pi^*  Z(g) = Z(g)
\end{align*}
in combination with the identity $Z(g \circ \pi \circ \rho_\alpha^{-1}) = Z(g)$. We conclude that $(\rho_{\alpha}^{-1})^* Z_H - Z ((\rho_{\alpha}^{-1})^*)$ is a vertical vector field, as desired.  
\endproof
Thus, the stabilisation map $v$ satisfies the hypotheses of Proposition \ref{C2frameconnection} 
and the associated Grassmann connection constitutes and example of a $C^{2}$-connection on a $C^{(1,2)}$-module.
\subsection{The universal lift}
We may use the isometry $v$ to obtain a convenient expression for the connection $\nabla^\X$. In particular, we can obtain 
a lift of $\nabla^\X$ to a universal connection $\nabla^\X_u$, 
where by `lift' we
mean that $\pi_{D_B} \circ \nabla^\X_u = c \circ \nabla^\X$
where $c$ is Clifford multiplication on spinors. 

Since $v^*v=1$ any $s \in \X$ can be written as $s = v^* ( F(s) )$ 
where $F(s) = v(s) \in C^\infty(F_0)^N \widehat \otimes C^\infty(B)$. 
For any $s\in \X$ there exist functions 
$f_k \in C^\infty(F_0)^N$, $g_k \in C^\infty(B)$ such that  
$$
F(s) =  \sqrt{\frac{d\mu}{d\mu_0}} \sum_k f_k \otimes g_k.
$$
Then we have
\begin{equation}
s = v^*(F(s)) = \sum_{\alpha ,n, k}(\rho_{\alpha}^{-1})^* \left(f_k \otimes g_k\right) \chi_\alpha x_{\alpha,n} 
\label{eq:ess-sum}
\end{equation}
so that with respect to local coordinates $\sigma_\alpha^\mu$ on $B$ we have
\begin{align*}
  \nabla^\X (s)& =\sum_{\alpha,n,k}  \nabla_{\partial/\partial \sigma_\alpha^\mu} ^\X \left((\rho_{\alpha}^{-1})^* \left(f_k \otimes g_k\right) \chi_\alpha x_{\alpha,n} \right ) \otimes_{C^\infty(B)} d \sigma_\alpha^\mu.
\end{align*}

Using the Leibniz rule and the fact that the derivative on the base commutes with the functions $f_k$ in the fibre direction, we find that
\begin{align}
&\nabla^\X (s)\label{eq:conn-pre-univ}\\
 &=\sum_{\alpha,n,k}  \nabla_{\partial/\partial \sigma_\alpha^\mu}^\X \left( \chi_\alpha x_{\alpha,n} \right ) (\rho_{\alpha}^{-1})^* \left(f_k \otimes 1\right) \otimes  g_k d \sigma_\alpha^\mu  + \sum_{\alpha,n,k}  \chi_\alpha x_{\alpha,n}  \cdot (\rho_{\alpha}^{-1})^* \left(f_k \otimes 1\right) \otimes  d g_k(s).\nonumber
\end{align}

\begin{lemma}
\label{lma:conn-univ-lift}
  The connection $\nabla^\X$ can be lifted to a universal connection
  $$
\nabla^\X_u : \X \to \X \widehat \otimes \Omega^1_u (C^\infty(B)), 
$$
in the sense that $\pi_{D_{B}} \circ \nabla^\X_u = c \circ \nabla^\X$ where $c$ denotes Clifford multiplication.
  \end{lemma}
\proof
From Equation \eqref{eq:conn-pre-univ} we identify 
a candidate universal connection as
\begin{equation}
  \nabla^\X_u (s)\! =\!\! \sum_{\alpha,n,k}  \nabla_{\partial/\partial \sigma_\alpha^\mu}^\X \left( \chi_\alpha x_{\alpha,n} \right ) (\rho_{\alpha}^{-1})^* \left(f_k \otimes 1\right) \otimes   \delta( \sigma_\alpha^\mu)g_k + \sum_{\alpha,n,k}  \chi_\alpha x_{\alpha,n}  \cdot (\rho_{\alpha}^{-1})^* \left(f_k \otimes 1\right) \otimes  \delta ( g_k).
  \label{eq:universal-conn}
\end{equation}
First of all, the right hand side of \eqref{eq:universal-conn} makes sense
in the projective tensor product topology.
The first sum is readily compared to \eqref{eq:ess-sum} since there
are only finitely many terms in the $\alpha,\,n$ sums. The second
term can similarly be compared to \eqref{eq:ess-sum} using the fact that
$\delta:\,C^\infty(B)\to\Omega^1_u(C^\infty(B))$ is completely bounded.

Multiplication of a section $s\in\X$ 
by a function $g \in C^\infty(B)$ via pullback along $\pi$ amounts to multiplying each $g_k$ by $g$. Applying the Leibniz rule
for $\delta$ to the right-hand side of 
Equation \eqref{eq:universal-conn} 
proves that $\nabla^\X_u(sg)=\nabla^\X_u(s)g+s\otimes\delta(g)$.
It is then clear by construction that $\pi_{D_B}\circ\nabla^\X_u$ coincides with $c\circ\nabla^\X$.
\endproof

In the next few statements we compare projective tensor products and Haagerup tensor products and so 
need the notation introduced in 
Equations \eqref{eq:the-map-known-as-iota}, \eqref{eq:iota-tee},
Section \ref{sec:uni-conn}. 

\begin{lemma}
\label{check}
The map
\[
\iota_{D_B}\circ\nabla_{u}^\X:\mathcal{X}\to X\otimes_{C^{1}(B)}^{h}\Omega^{1}_{D_{B}}(C^{1}(B)),
\]
is continuous.
\end{lemma}
\begin{proof} For $x\in\mathcal{X}$ we have an equality 
\begin{align*}
\iota_{D_B}\circ\nabla_{u}^\X (x)=\nabla_{D_B} (x)&=(1\otimes_{\nabla}D_B|x\rangle - |\gamma(x)\rangle D_B )\\
&=v^{*}(v(1\otimes_{\nabla}D_B)-(\gamma\otimes D_{B})v)(x) +v^{*}\left[\gamma\otimes D_B, v(x)\right],
\end{align*}
of operators $\Dom D_{B}\to X\otimes^{h}_{C(B)}L^{2}(\mathcal{E}_{B})$. Since
\begin{align*}
v^{*}(v(1\otimes_{\nabla}D_B)-(\gamma\otimes D_{B})v)=v^{*}(v(1\otimes_{\nabla}D_B)-(\gamma\otimes D_{B})v)(D_{V}+i)^{-1}(D_{V}+i),
\end{align*}
and $D_{V}:\mathcal{X}\to\mathcal{X}$ is continuous,  this operator is continuous by Proposition \ref{prop: quasicomm}.
Furthermore,
\[
v^{*}\left[\gamma\otimes D_B, v(x)\right]=v^{*}(\gamma\otimes c(\d v(x))),
\]
and $x\mapsto c(\d v(x))$ is a composition of continuous maps
\[\mathcal{X}\xrightarrow{v} C^{\infty}(F_0)\widehat{\otimes}C^{\infty}(B)\xrightarrow{1\otimes \d} C^{\infty}(F_0)\widehat{\otimes}\Omega^{1}(B)\xrightarrow{c} L^{2}(F_0)^{N}\otimes^{h}\Omega^{1}_{D_{B}}(C^{1}(B)).\] 
Since $v^{*}:L^{2}(F_0)^{N}\otimes^{h}\Omega^{1}_{T}(\mathcal{B}_{1})\to X\otimes^{h}_{B}\Omega^{1}_{D_B}(\mathcal{B}_{1})$ is continuous as well, the lemma is proved.
\end{proof}
\begin{corl}
  \label{corl:riem-subm-C2}
The pair $(\mathcal{X},D_{V})$ is a $C^{2}$-module relative to $(C^{2}(B), L^{2}(\mathcal{E}_{B}), D_{B})$ and the universal connection 
$$
\nabla_u^\X :\mathcal{X}\to \mathcal{X}\widehat{\otimes}_{C^{\infty}(B)}\Omega^{1}_{u}(C^{\infty}(B)),
$$
defines a $C^{2}$-connection $(\nabla,\nabla^{1})$  on $(\mathcal{X},D_V)$, where $\nabla=\iota_0\circ\nabla_u^\X$ and $\nabla^1=\iota_1\circ\nabla_u^\X$. The quintuple $(C^{2}(M),\mathcal{X}\otimes^{\alg}_{\mathcal{B}}\mathcal{B}_{2}, D_{V},(\nabla,\nabla^{1}))$ is a $C^{2}$-correspondence 
for the spectral triples $(C^{2}(M), L^{2}(\mathcal{E}_M), D_{M})$ and $(C^{2}(B), L^{2}(\mathcal{E}_{B}), D_{B})$.
  \end{corl}
 \begin{proof}
  Since $\langle \X,\X\rangle\subset C^{\infty}(B)$ 
  and $D_{V}:\mathcal{X}\to \mathcal{X}\subset X$ is 
  essentially self-adjoint and regular, $(\mathcal{X},D_{V})$ is a 
  $C^{2}$-module (see Definition \ref{def: horver}). 
Next we show how to obtain a pair $(\nabla,\nabla^1)$ satisfying Definition \ref{def: connpair}. As the inclusions $\mathcal{X}\to X$ and $C^{\infty}(B)\to \textnormal{Lip}(D_{B})$ are continuous, we obtain a connection 
$\nabla:=\iota_0\circ\nabla_{u}^\X:\mathcal{X}\to X\otimes_{B}\Omega^{1}(C^{1}(B))$. 
By  Lemma \ref{lma:commutator-Kuc} the pair $(D_{V}\otimes 1, 1\otimes_{\nabla}D_{B})$ is vertically anti-commuting and by Lemma \ref{check} 
  \[
\iota_{D_B}\circ\nabla_{u}^\X:\mathcal{X}\to X\otimes_{C^{1}(B)}^{h}\Omega^{1}_{D_{B}}(C^{1}(B),C(B)),
\]
is continuous. 
The conclusion now follows from Equation \eqref{eq:factorization} and Corollary \ref{universalsquare}.
   \end{proof}
 
\subsection{Curvature of $\nabla$}

In this section we compute the curvature $R_{\nabla^\X_{u}}$ of the connection $\nabla^\X$, which is given by
$$
R_{\nabla^\X_{u}} = (1\otimes_{\nabla_u^\X} D_B^2) - (1 \otimes_{\nabla_u^\X} D_B)^2
$$
in terms of the Dirac operator $D_B$ on the base manifold of the submersion.

\begin{prop}
  \begin{enumerate}
    \item
  The curvature operator $R_{\nabla^\X} \equiv \pi_{D_B}((\nabla^\X_u)^2)$ 
  on $\X \widehat \otimes_{C^\infty(B)} \E_B$ 
  is given by Clifford multiplication with the curvature of $\nabla^\X$. 
  More precisely, in terms of a local orthonormal frame $\{ f_j\}$ of 
  vector fields on $B$ we have the equality
  $$
R_{\nabla^\X_{u}} =c \circ \left((\nabla^\X)^2\right) \equiv \sum_{j,k}\left(\left[ \nabla^\X_{f_j}, \nabla^\X_{f_k} \right]- \nabla^\X_{[f_j,f_k]} \right) \gamma_j \gamma_k
$$
as skew-symmetric operators from $\X \otimes_{C^\infty(B)}^{\rm alg} \E_B$ to $X \otimes_{C(B)} L^2(\E_B)$ and where $\gamma_j = c((f_j)_H)$ are flat Dirac matrices.
\item 
There is the following local expression for the curvature in terms of the curvature $\Omega$ of the submersion and the connection one-form $A^\X$ of $\nabla^X$:
$$
R_{\nabla^\X_{u}}(\xi) = \sum_{j,k} \left( \sum_i \Omega(\cdot,\cdot ,e_i) e_i + \d A^\X + (A^\X)^2 \right) ((f_j)_H,(f_k)_H) \gamma_j \gamma_k \xi
$$
with $\xi$ supported in a suitable coordinate chart of $M$ and where $\{e_i\} $ is a local orthonormal frame of vertical vector fields on $M$. 

\item The curvature operator $R_{\nabla^\X_{u}}$ is relatively $D_V$-bounded.
  \end{enumerate}
\end{prop}
\proof
In view of Corollary \ref{universalsquare} 
we have that 
$R_{\nabla^\X} \equiv \pi_{D_B}((\nabla^\X_u)^2)$. 
By Lemma \ref{lma:conn-univ-lift} and the fact that the Clifford representation of universal forms factors through the DeRham calculus,  we may compute the 
right-hand side by working with $\pi_D$-represented de Rham differential forms and thus exploit local expressions. Let us start by writing the connection $\nabla^{\X}$ in terms of a connection one-form: using Definition \ref{defn:metric-conn} we have for $\xi \in \X$ supported in a trivializing chart (for $\E_V$) on $M$:
  $$
\nabla^{\X}_Z (\xi) = Z_H(\xi) + A^{\E_V} (Z_H)(\xi) + \frac 12 k(Z_H) \cdot \xi 
  $$
where we have written $\nabla^{\E_V} = \d_M + A^{\E_V}$ in terms of a (locally-defined) connection one-form $A^{\E_V} \in \End_{C^\infty(M)} (\E_V) \otimes_{C^\infty(M)} \Omega^1(M)$. 

Let us define a combined connection one-form as $A^{\X} := A^{\E_V} + \frac 12 k\in \End_{C^\infty(M)} (\E_V) \otimes_{C^\infty(M)} \Omega^1(M)$ so that $\nabla^{\X}_Z  = Z_H + A^\X(Z_H)$. Note that since $M$ is compact we can assume that there is a finite number of such trivialising charts. Hence, for the relative bounds of the curvature operator that we are after here we may just as well work on a single chart. 

With these preparations, we compute the curvature operator 
acting on a $\xi \in\X$ supported in a single chart, finding 
\begin{align*}
  c \circ \left( (\nabla^\X)^2\right)(\xi) &=  \sum_{j,k}\left(\left[ \nabla^\X_{f_j}, \nabla^\X_{f_k} \right]- \nabla^\X_{[f_j,f_k]} \right) \gamma_j \gamma_k (\xi)\\
  &= \sum_{j,k} \left( [(f_j)_H,(f_k)_H ] - [f_j,f_k]_H \right) \gamma_j \gamma_k   (\xi)  + c_H \circ (\d A^\X + (A^\X)^2) (\xi),
  \end{align*}
where $c_H$ denotes Clifford multiplication only in the horizontal direction (involving the $\gamma_j$ and  $\gamma_k$). This last term satisfies
\begin{align*}
&   \langle c_H (\d A^\X + (A^\X))  \xi ,  c_H (\d A^\X + (A^\X)) \xi \rangle_X(b)  \\
  & \qquad = \int_{\pi^{-1}(b)} \langle c_H (\d A^\X + (A^\X))  \xi ,  c_H (\d A^\X + (A^\X)) \xi\rangle_{\E_V} (x) \d\mu_{\pi^{-1}(b)} (x) \\
  & \qquad \leq \| c_H(\d A^\X + (A^\X))\|^2_{\End_{C^\infty(M)} (\E_V)} \langle \xi, \xi \rangle (b)
\end{align*}
The relevant and potentially unbounded term in the 
curvature is thus $[(f_j)_H,(f_k)_H ] - [f_j,f_k]_H$. 
But this difference of commutators is a vertical vector field and, in fact, 
it is precisely the one described by the curvature $\Omega$ of $\pi$ as defined in Definition \ref{defn:riem-subm-curv}. Indeed, the horizontal lift of a commutator is the horizontal part of the commutator of the lifted vector fields and hence 
$$
[X_H,Y_H ] - [X,Y]_H = \sum_i \Omega(X_H, Y_H , e_i) e_i
$$
for vector fields $X,Y$ on $B$ and an orthonormal frame $\{e_i\}$ of vertical vector fields. 
We conclude that the curvature is given locally by a vertical vector field plus bounded terms, and since $D_V$ is vertically elliptic we find the desired relative bound. 
\endproof


\end{document}